\documentclass[reqno]{amsart}
\usepackage[top=2cm,bottom=2cm,right=2.5cm,left=2.5cm]{geometry}
\usepackage{amssymb}
\usepackage{amsmath, amsthm, amscd, amsfonts, amssymb, graphicx, color}
\usepackage[bookmarksnumbered, colorlinks, plainpages]{hyperref}
\hypersetup{colorlinks=true,linkcolor=red, anchorcolor=green, citecolor=cyan, urlcolor=red, filecolor=magenta, pdftoolbar=true}
 
\usepackage{hyperref}

\newtheorem{theorem}{Theorem}[section]
\newtheorem{lemma}[theorem]{Lemma}
\newtheorem{proposition}[theorem]{Proposition}
\newtheorem{corollary}[theorem]{Corollary}
\theoremstyle{definition}
\newtheorem{definition}[theorem]{Definition}
\newtheorem{example}[theorem]{Example}

\theoremstyle{remark}
\newtheorem{remark}[theorem]{Remark}

\numberwithin{equation}{section}

\begin{document}
	\title[Continuous Frame in Hilbert $C^{\ast}$-Modules]{Continuous Frame in Hilbert $C^{\ast}$-Modules}
\author[M. Rossafi, M. Ghiati, M. Mouniane, F. Chouchene, S. Kabbaj]{Mohamed Rossafi$^{1}$, M'hamed Ghiati$^{2}$, Mohammed Mouniane$^{2*}$, Frej Chouchene$^{3}$ \MakeLowercase{and} Samir Kabbaj$^{4}$}

\address{$^{1}$LaSMA Laboratory Department of Mathematics Faculty of Sciences, Dhar El Mahraz University Sidi Mohamed Ben Abdellah, P. O.  Box 1796 Fez Atlas, Morocco}
\email{\textcolor[rgb]{0.00,0.00,0.84}{rossafimohamed@gmail.com}}

\address{$^{2}$Laboratory Analysis, Geometry and Applications Department of Mathematics, Faculty of Sciences, University of Ibn Tofail, P. O.  Box 133 Kenitra, Morocco}
\email{\textcolor[rgb]{0.00,0.00,0.84}{mhamed.ghiati@uit.ac.ma; mouniane.mohammed@uit.ac.ma}}

\address{$^{3}$ Department of Mathematics, Higher School of Schiences and Technology of Hammam Sousse, University of Sousse, Tunisia}
\email{\textcolor[rgb]{0.00,0.00,0.84}{frej.chouchene@essths.u-sousse.tn}}

\address{$^{4}$Laboratory of Partial Differential Equations, Spectral Algebra and Geometry Department of Mathematics, Faculty of Sciences, University of Ibn Tofail, P. O.  Box 133 Kenitra, Morocco}
\email{\textcolor[rgb]{0.00,0.00,0.84}{samkabbaj@yahoo.fr}}


\address{$^*$Corresponding author}

\begin{abstract}
	Frame theory is an exciting, dynamic and fast paced subject with applications in numerous fields of mathematics and engineering. In this paper we study Continuous Frame and introduce Continuous Frame with $C^{\ast}$-valued bounds. Also, we establich some properties.
\end{abstract}

\maketitle

\section{Introduction and preliminaries}
The concept of frames in Hilbert spaces has been introduced by Duffin and Schaeffer \cite{Duffin-1952-TransAmerMathSoc-1952} in 1952 to study some deep problems in nonharmonic Fourier series, after the fundamental paper \cite{Daubechies-JMATHPHYS-1986} by Daubechies, Grossman and Meyer, frame theory began to be widely used, particularly in the more specialized context of wavelet frames and Gabor frames \cite{Gabor-J.Electr.Eng-1946}.

Traditionally, frames have been used in signal processing, image processing, data compression, and sampling theory. 
A discreet frame is a countable family of elements in a separable Hilbert space which allows for a stable, not necessarily unique, decomposition of an arbitrary element into an expansion of the frame elements. 
The concept of a generalization of frames to a family indexed by some locally compact space endowed with a Radon measure was proposed by G. Kaiser \cite{Kaiser-book-2011} and independently by Ali,
Antoine and Gazeau \cite{Ali-Antoine-1993-AnnPhys}. These frames are known as continuous frames. Gabardo and
Han in \cite{Gabardo-Adv.Comput-2003} called these  frames associated with measurable spaces, Askari-Hemmat,
Dehghan and Radjabalipour in \cite{Askari-Hemmat-PAMS-2000} called them generalized frames and in mathematical physics, they are referred to as coherent states \cite{Ali-Antoine-1993-AnnPhys}.

A discrete frame in a Hilbert $C^{\ast}$-module $\mathcal{H}$ is a sequence $\{f_{i}\}_{i\in I}$  for which there exist positive constants $A, B > 0$ called frame bounds such that
\begin{equation*}
	A\langle x, x\rangle\leq\sum_{i\in I}\langle x,f_{i}\rangle\langle f_{i}, x\rangle\leq B\langle x, x\rangle, \;\forall x\in\mathcal{H}.
\end{equation*}
Many generalizations of the concept of frame have been defined in Hilbert Spaces and Hilbert $C^{\ast}$-modules \cite{Rossafi-Kabbaj-2020-Asia-Europ, Rossafi-Bour-2019-Asia, Rossafi-Kabbaj-2019-Lin.Algebra, Rossafi-Kabbaj-AnnUnivPead-2018, Rossafi-Kabbaj-2018-Asia-Math, Rossafi-Akhlidj-Math-Recherche-2018}.

In this paper, we introduce the notions of Continuous Frame on a Hilbert $C^{\ast}$-Modules over an unital $C^{\ast}$-algebra which is a generalization of discrete frames, the  $\ast$-Continuous Frame which is a generalization of $\ast$-Frame in Hilbert $C^{\ast}$-Modules introduced by A. Alijani, M. Dehghan \cite{Alijani-Dehghan-2011} and we establish some new results.

The paper is organized as follows, we continue this introductory section we briefly recall the definitions and basic properties of Hilbert $C^{\ast}$-modules. In Section 2, we introduce the Continuous Frame, the pre-Continuous frame operator and the Continuous frame operator. In Section 3, we introduce the $\ast$-Continuous frame and the $\ast$-Continuous frame operator. In Section 4, we discuss the stability problem for Continuous Frame and $\ast$-Continuous frame.
In section 5, we introduce dual $*$-frames and extend the characterization of dual frames \cite{Christensen-Book} to dual *-frames associated to a given $*$-frame.
In the following we briefly recall the definitions and basic properties of Hilbert $C^{\ast}$-modules. Our reference for $C^{\ast}$-algebras is \cite{Davidson-1996}. For a $C^{\ast}$-algebra $\mathcal{A}$ if $a\in\mathcal{A}$ is positive we write $a\geq 0$ and $\mathcal{A}^{+}$ denotes the cone of positive elements of $\mathcal{A}$.
\begin{definition}[\cite{Kaplansky-1953-AmJMath}]
	Let $ \mathcal{A} $ be a unital $C^{\ast}$-algebra and $\mathcal{H}$ be a left $ \mathcal{A} $-module, such that the linear structures of $\mathcal{A}$ and $ \mathcal{H} $ are compatible. $\mathcal{H}$ is a pre-Hilbert $\mathcal{A}$-module if $\mathcal{H}$ is equipped with an $\mathcal{A}$-valued inner product $\langle.,.\rangle :\mathcal{H}\times\mathcal{H}\rightarrow\mathcal{A}$, such that is sesquilinear, positive definite and respects the module action. In other words,
	\begin{itemize}
		\item [1.] $ \langle x,x\rangle\geq0 $ for all $ x\in\mathcal{H} $ and $ \langle x,x\rangle=0$ if and only if $x=0$.
		\item [2.] $\langle ax+y,z\rangle=a\langle x,z\rangle+\langle y,z\rangle$ for all $a\in\mathcal{A}$ and $x,y,z\in\mathcal{H}$.
		\item[3.] $ \langle x,y\rangle=\langle y,x\rangle^{\ast} $ for all $x,y\in\mathcal{H}$.
	\end{itemize}	 
\end{definition}
For $x\in\mathcal{H}, $ we define $\|x\|=\|\langle x,x\rangle\|^{\frac{1}{2}}$. If $\mathcal{H}$ is complete with $\|\cdot\|$, it is called a Hilbert $\mathcal{A}$-module or a Hilbert $C^{\ast}$-module over $\mathcal{A}$. 
For every $a$ in $C^{\ast}$-algebra $\mathcal{A}$, we have $\mid a\mid=(a^{\ast}a)^{\frac{1}{2}}$ and the $\mathcal{A}$-valued norm on $\mathcal{H}$ is defined by $\mid x \mid =\langle x, x\rangle^{\frac{1}{2}}$ for $x\in\mathcal{H}$.

Let $\mathcal{H}$ and $\mathcal{K}$ be two Hilbert $\mathcal{A}$-modules. A map $T:\mathcal{H}\rightarrow\mathcal{K}$ is said to be adjointable if there exists a map $T^{\ast}:\mathcal{K}\rightarrow\mathcal{H}$ such that $\langle Tx,y\rangle_{\mathcal{A}}=\langle x,T^{\ast}y\rangle_{\mathcal{A}}$ for all $x\in\mathcal{H}$ and $y\in\mathcal{K}$.

We also reserve the notation $End_{\mathcal{A}}^{\ast}(\mathcal{H},\mathcal{K})$ for the set of all adjointable operators from $\mathcal{H}$ to $\mathcal{K}$ and $End_{\mathcal{A}}^{\ast}(\mathcal{H},\mathcal{H})$ is abbreviated to $End_{\mathcal{A}}^{\ast}(\mathcal{H})$.

The following lemmas will be used to prove our main results.
\begin{lemma}[\cite{Paschke-1973-TransAmMath}]
	\label{1}
	Let $\mathcal{H}$ be Hilbert $\mathcal{A}$-module. If $T\in End_{\mathcal{A}}^{\ast}(\mathcal{H})$, then $$\langle Tx,Tx\rangle\leq\|T\|^{2}\langle x,x\rangle, \qquad x\in\mathcal{H}.$$
\end{lemma}
\begin{lemma}[\cite{Aramba-2007-ProcAmMathSoc}] 
	\label{l2} 
	Let $\mathcal{H}$ and $\mathcal{K}$ two Hilbert $\mathcal{A}$-modules and $T\in End^{\ast}(\mathcal{H},\mathcal{K})$. 
	So the following statements are equivalent:
	\begin{itemize}
		\item [(i)] $T$ is surjective.
		\item [(ii)] $T^{\ast}$ is bounded below with respect to norm, i.e., there is $m>0$ such that $ m\|x\| \leq \|T^{\ast}x\|$ for all $x\in\mathcal{K}$.
		\item [(iii)] $T^{\ast}$ is bounded below with respect to the inner product, i.e., there is $m'>0$ such that $m'\langle x,x\rangle \leq \langle T^{\ast}x,T^{\ast}x\rangle $ for all $x\in\mathcal{K}$.
	\end{itemize}
\end{lemma} 
\begin{lemma}[\cite{Alijani-Dehghan-2011}] \label{Lemma0-1} 
	\label{3} 
	Let $\mathcal{H}$ and $\mathcal{K}$ two Hilbert $\mathcal{A}$-modules and $T\in End^{\ast}(\mathcal{H},\mathcal{K})$. Then:
	\begin{itemize}
		\item [(i)] If $T$ is injective and $T$ has a closed range, then the adjointable map $T^{\ast}T$ is invertible and $$\|(T^{\ast}T)^{-1}\|^{-1}\leq T^{\ast}T\leq\|T\|^{2}.$$
		\item  [(ii)]	If $T$ is surjective, then the adjointable map $TT^{\ast}$ is invertible and $$\|(TT^{\ast})^{-1}\|^{-1}\leq TT^{\ast}\leq\|T\|^{2}.$$
	\end{itemize}	
\end{lemma}
\begin{lemma}[\cite{Yosida-1978-FunctionalAnalysis}] 
	Let $(\varOmega ,\mu)$ be a measure space, $X$ and $Y$ are tow Banach spaces, $\lambda : X \rightarrow Y$ be a bounded linear operator and $ f :\varOmega \rightarrow Y $ measurable function, then $$\lambda (\int_\varOmega f d\mu)	=\int_\varOmega(\lambda f)d\mu.$$
\end{lemma}
\begin{lemma}[\cite{Xu2008}]
	Let $ T \in End^{*}_{\mathcal{A}}(\mathcal{H},\mathcal{K}) $ be  a bounded operator with closed range  $\mathcal{R} (T)$.  
	Then there exists a bounded operator $ T^{\dagger} \in End^{*}_{\mathcal{A}}( \mathcal{K}, \mathcal{H}) $ for which 
	$$ TT^{\dagger}x=x      ,  \qquad x\in\mathcal{R} (T). $$
\end{lemma}
\begin{proposition}[\cite{Murphy-1990}]\label{Proposition-0-1}
	If $\varphi: \mathcal{A} \longrightarrow \mathcal{B}$ is a $*$-homomorphism between $C^{*}$-algebras, then $\varphi$ has the following properties.
	\begin{enumerate}
		\item  $\varphi(1)=1$.
		\item  If a is invertible, then so is $\varphi(a)$, and $\varphi\left(a^{-1}\right)=\varphi(a)^{-1}$.
		\item The $*$-homomorphism $\varphi$ is positive and increasing, that is, $\varphi\left(\mathcal{A}^{+}\right) \subseteq \mathcal{B}^{+}$, and if $a_{1} \leq$ $a_{2}$, then $\varphi\left(a_{1}\right) \leq \varphi\left(a_{2}\right)$.
		\item  For $a \in \mathcal{A}$, we have $\sigma(\varphi(a)) \subseteq \sigma(a)$, and if $\varphi$ is injective, then $\sigma(\varphi(a))=\sigma(a)$.
		\item  If a is strictly positive, then so is $\varphi(a)$.
	\end{enumerate}
	
\end{proposition}
\section{Continuous Frame in Hilbert $C^{\ast}$-modules over a unital $C^{\ast}$-algebra}
Let $X$ be a Banach space, $(\Omega,\mu)$ a measure space and a measurable function $f:\Omega\to X$. Integral of the Banach-valued function $f$ has been defined by Bochner and others. Most properties of this integral are similar to those of the integral of real-valued functions. Because every $C^{\ast}$-algebra and Hilbert $C^{\ast}$-module is a Banach space thus we can use this integral and its properties.

\noindent
Let $(\Omega,\mu)$ be a measure space, we define
\begin{equation*}
	L^{2}(\Omega, \mathcal{A})=\Bigg\{\varphi: \Omega \to \mathcal{A}: \bigg\|\int_{\Omega}\varphi(\omega)\varphi(\omega)^{\ast}d\mu(\omega)\bigg\|<\infty\Bigg\}.
\end{equation*}
For any $\varphi, \psi \in L^{2}(\Omega, \mathcal{A})$, if the $\mathcal{A}$-valued inner product is defined by \begin{equation*}\langle \varphi ,\psi\rangle=\int_{\Omega} \varphi(\omega)\psi(\omega)^{\ast} d\mu(w),
\end{equation*} 
the norm is defined by $\|\varphi\|=\|\langle \varphi, \varphi\rangle\|^{\frac{1}{2}}$, then $L^{2}(\Omega, \mathcal{A})$ is a Hilbert $C^{\ast}$-module \cite{lance}.

The following definition generalize the Definition 2.1. in \cite{Rahimi-Najati-Dehghan-2006} to the context of Hilbert $C^{\ast}$-module. 
\begin{definition}
	Let $\mathcal{H}$ be a Hilbert $\mathcal{A}$-module and $(\Omega,\mu)$ a measure space. A mapping $F: \Omega \to \mathcal{H}$ is called a continuous frame with respect to $(\Omega,\mu)$, if
	\begin{enumerate}
		\item   for all $x\in\mathcal{H}, w \to \langle x, F_{w}\rangle $ is a measurable function on $\Omega$,
		\item   there exist constants $A, B>0$ such that
		\begin{equation} \label{2.1}
			A\langle x,x\rangle
			\leq\int_{\Omega}\langle x, F_{w}\rangle\langle F_{w}, x\rangle d\mu(w)\leq B\langle x,x\rangle, \forall x\in\mathcal{H}.
		\end{equation}
	\end{enumerate}
	The constants $A$ and $B$ are called continuous frame bounds. If $A=B$ we call this continuous frame a continuous tight frame, and if $A=B=1$ it is called a Parseval continuous frame. If only the right-hand inequality of \eqref{2.1} is satisfied, we call  $F: \Omega \to \mathcal{H}$ a continuous Bessel mapping with Bessel bound $B$.
\end{definition}
\begin{example}
	Let $\mathcal{A}$ be the $C^{\ast}$-algebra $\mathbb{M}_{2,2}(\mathbb{C})$ and $\mathcal{H}$ be the Hilbert $\mathbb{M}_{2,2}(\mathbb{C})$-module $\mathbb{M}_{2,2}(\mathbb{C})$.
	
	\noindent
	Let $(\Omega,\mu)=([0, 1], d\lambda)$ where $d\lambda$ is the measure of Lebesgue,
	
	\noindent
	define $F: \Omega \to \mathcal{H}$ by $F_{w}=\begin{bmatrix}
		w&0\\0&w
	\end{bmatrix}$.
	
	So 
	\begin{equation*}
		\int_{\Omega}\langle T, F_{w}\rangle\langle F_{w}, T\rangle d\lambda=\int_{0}^{1}TF_{w}^{\ast}F_{w}T^{\ast}d\lambda
		=\int_{0}^{1}w^{2}TT^{\ast}d\lambda
		=\frac{1}{3}TT^{\ast}
		=\frac{1}{3}\langle T, T\rangle.
	\end{equation*}
	Then $F$ is a tight continuous frame for $\mathcal{H}$ with respect to $([0, 1], d\lambda)$.
\end{example}
Suppose that $F$ is a continuous frame for $\mathcal{H}$ with respect to $(\Omega,\mu)$. The operator $T_{F}: \mathcal{H} \to L^{2}(\Omega, \mathcal{A})$ defined by $T_{F}x(\omega)=\langle x, F_{\omega}\rangle, \omega\in\Omega$, is called the analysis operator. $T_{F}$ is adjointable and $T_{F}^{\ast}: L^{2}(\Omega, \mathcal{A}) \to \mathcal{H}$ is given by $T_{F}^{\ast}\varphi=\int_{\Omega}\varphi(\omega)F_{\omega} d\mu(\omega)$, is called the synthesis operator.
\begin{proposition}
	Let $ F$ be a continuous frame with respect to $ (\Omega,\mu )$ for $\mathcal{H}$ with bounds $A,B$ and let $ T :\mathcal{H}  \rightarrow \mathcal{K} $ be a bounded operator with a closed range $ R_{T}$. Then $ TF $ is a continuous frame for $R_{T}$ with the bounds $ A\lVert T^{\dagger}\lVert^{-2},B\lVert T\lVert^{2}$.
\end{proposition}
\begin{proof}
	It is clear that $ w\rightarrow V F(w)$  is measurable for all $ f \in \mathcal{H} $.   We may assume that
	$T$ is onto. If $x \in \mathcal{K}$, then
	\begin{equation*}
		\int_{\Omega}\langle x,TF_{w}\rangle\langle TF_{w}, x\rangle d\mu(w)\leq B\langle Tx,Tx\rangle \leq B\lVert T\lVert^{2}\langle x,x\rangle.  
	\end{equation*}
	which proves that $TF_{w}$ is Bessel. For the lower frame condition, let $ f\in K$. Then $$ \lVert x\lVert \leq\lVert T^{\dagger}\lVert\lVert T^{*}x\lVert $$ and
	\begin{equation*}
		\int_{\Omega}\langle x,TF_{w}\rangle\langle TF_{w}, x\rangle d\mu(w)\geq A\langle T^{*}x,T^{*}x\rangle \geq A\lVert T^{\dagger}\lVert^{-2}\langle x,x\rangle  ,
	\end{equation*}
	which gives the result.	
\end{proof}
\begin{corollary}
	If $F$ is a continuous frame with respect to $(\Omega
	, \mu)$ for $\mathcal{H}$ with bounds
	$A,B$ and $ T : \mathcal{H} \rightarrow \mathcal{K}$ is a bounded surjective operator, then $TF$ is a continuous frame
	with respect to $(\Omega
	, \mu)$ for $ \mathcal{K}$ with the bounds $ A\lVert T^{\dagger}\lVert^{-2},B\lVert T\lVert^{2}$.
\end{corollary}
\begin{theorem}\label{thm-Operator-TF}
	Let $(\Omega
	, \mu)$ be a measure space and let $F_{w}$ be a Bessel mapping from $\Omega$
	to $ \mathcal{H}$. Then the operator $ T_{F} : L^2(\Omega
	, \mu) \rightarrow \mathcal{H} $ weakly defined by
	$$\langle T_{F}\varphi,x\rangle =\int_{\Omega}\varphi(w)\langle F_{w} x,x\rangle d\mu(w) ,\qquad x \in \mathcal{H}.  $$
	is well defined, linear, bounded, and its adjoint is given by $$T^{\ast}_{F} : \mathcal{H} \rightarrow  L^2(\Omega
	, \mu) ,(T^{\ast}_{F}x)(w)=\langle x,F_{w}\rangle, \qquad w \in\Omega.  $$
	The operator $T_{F}$ is called a pre-frame operator or synthesis operator and $ T^{\ast}_{F}
	$ is called an analysis operator of $ F$.\\
	\begin{proof}
		The proof is straightforward. 
	\end{proof}		
	The converse of Theorem \ref{thm-Operator-TF} holds when $\mu $ is a finite measure.
\end{theorem}
\begin{definition}
	Let $F$ be a continuous frame for $\mathcal{H}$ with respect to $(\Omega,\mu)$. We define the frame operator $S:\mathcal{H}\to\mathcal{H}$ by 
	$$Sx=T_{F}^{\ast}T_{F}x=\int_{\Omega}\langle x, F_{w}\rangle F_{w} d\mu(w), \forall x\in\mathcal{H}.$$
\end{definition}
\begin{theorem}
	The frame operator  $S$ is positive, self-adjoint and invertible.
\end{theorem}
\begin{proof}
	Let $x, y\in\mathcal{H}$, we have
	\begin{align*}
		\langle Sx, y\rangle&=\Big\langle\int_{\Omega}\langle x, F_{w}\rangle F_{w} d\mu(w), y\Big\rangle\\&=\int_{\Omega}\langle x, F_{w}\rangle\langle F_{w}, y\rangle d\mu(w)\\&=\Big\langle x, \int_{\Omega}\langle y, F_{w}\rangle F_{w}d\mu(w)\Big\rangle\\&=\langle x, Sy\rangle,
	\end{align*}
	so the operator $S$ is self-adjoint.
	
	\noindent
	Let $x\in\mathcal{H}$, by the definition of a continuous frame for $\mathcal{H}$ we have
	\begin{equation*}
		A\langle x,x\rangle
		\leq\int_{\Omega}\langle x, F_{w}\rangle\langle F_{w}, x\rangle d\mu(w)\leq B\langle x,x\rangle.
	\end{equation*}
	So
	\begin{equation}\label{1.}
		A\langle x,x\rangle\leq\langle Sx,x\rangle\leq B\langle x,x\rangle.
	\end{equation}
	Thus $S$ is positive and invertible.
\end{proof}
\begin{proposition}\label{Proposition-2.7}
	Let $(\Omega, \mu)$ 
	be a measure space, where $ \mu$ is a $\sigma$ -finite measure and let $
	F :\Omega 
	\rightarrow \mathcal{H}$ be a measurable function. If the mapping $ T_{F} : L^2(\Omega
	, \mu) \rightarrow \mathcal{H} $ defined by
	$$\langle T_{F}\varphi,h\rangle =\int_{\Omega}\varphi(w)\langle F_{w} x,h\rangle d\mu(w) ,\qquad h \in  L^2(\Omega
	, \mu)  $$
	is a bounded operator, then F is Bessel.
\end{proposition}
\begin{proof}
	By Theorem \ref{thm-Operator-TF}, we have
	$$(T^{\ast}h)(w) = \langle h,F \rangle. $$ 
	Hence, for each $ h \in \mathcal{H}$,
	\begin{equation*}
		\int_{\Omega}\langle h,F_{w}\rangle\langle F_{w}, h\rangle d\mu(w) = \lVert Th\lVert^{2} \leq\lVert h \lVert^{2} \lVert T\lVert^{2}.  
	\end{equation*}
\end{proof}
\begin{theorem}
	Let $(\Omega
	, \mu)$ be a measure space where $\mu$ is a $\sigma$ finite measure. The mapping
	$F : 
	\rightarrow H $ is a continuous frame with to $(\Omega
	, \mu)$ for $\mathcal{H}$ if and only if the operator
	TF as defined in Theorem \ref{thm-Operator-TF} is a bounded and onto operator.
\end{theorem}
\begin{proof}
	Let $F$ be a continuous frame. Then, by Theorem \ref{thm-Operator-TF}, $T_{F}$ is bounded and
	$$T^{\ast}_{F} : \mathcal{H} \rightarrow  L^2(\Omega
	, \mu) ,(T^{\ast}_{F}x)(w)=\langle x,F_{w}\rangle ,\qquad w \in\Omega. $$
	Hence, for each $x \in \mathcal{H}$
	\begin{equation*}
		\lVert T^{\ast}_{F}x\lVert^{2} =	\int_{\Omega}\langle x,F_{w}\rangle\langle F_{w}, x\rangle d\mu(w)
	\end{equation*}
	is one to one and so $T_{F}$ is onto.
	Conversely, let $T^{\ast}_{F}$ be a bounded and onto operator. Then,  there exists
	a bounded operator $ T^{\dagger}_{F}$
	such that $T_{F} T^{\dagger}_{F}
	x = x$ for all $x \in \mathcal{H}.$ Since $T_{F}$ is
	bounded, by  {Proposition \ref{Proposition-2.7}}, $ F $ is Bessel and
	\begin{equation*}
		\lVert T^{\ast}_{F}x\lVert^{2} =	\int_{\Omega}\langle x,F_{w}\rangle\langle F_{w}, x\rangle d\mu(w) \qquad x\in \mathcal{H}.
	\end{equation*}
	$ \lVert x\lVert^{2} \leq \lVert T^{\ast}_{F}x\lVert^{2} \lVert T^{\dagger}_{F}
	\lVert^{2}, \qquad x \in \mathcal{H}.$
	\begin{equation*}
		\lVert T^{\dagger}_{F}\lVert^{-2}	\lVert x\lVert^{2}
		\leq\int_{\Omega}\langle x, F_{w}\rangle\langle F_{w}, x\rangle d\mu(w), \qquad x \in \mathcal{H}.
	\end{equation*}
\end{proof}
\begin{theorem}
	Let $F$ be a continuous frame for $\mathcal{H}$ with respect to $(\Omega,\mu)$ with the frame operator $S$. Let $V\in End_{\mathcal{A}}^{\ast}(\mathcal{H},\mathcal{K})$ be a surjective operator. Then $VF$ is a continuous frame for $\mathcal{K}$ with the frame operator $VSV^{\ast}$.
\end{theorem}

\begin{proof}
	The mapping $VF: \Omega \to \mathcal{K}$ is measurable.
	Therefore,
	\begin{equation*}
		A\langle V^{\ast}x,V^{\ast}x\rangle
		\leq\int_{\Omega}\langle V^{\ast}x, F_{w}\rangle\langle F_{w}, V^{\ast}x\rangle d\mu(w)\leq B\langle V^{\ast}x,V^{\ast}x\rangle, \forall x\in\mathcal{K}.
	\end{equation*}
	So
	\begin{equation*}
		A\|(VV^{\ast})^{-1}\|^{-1}\langle x,x\rangle
		\leq\int_{\Omega}\langle x, VF_{w}\rangle\langle VF_{w}, x\rangle d\mu(w)\leq B\|V\|^{2}\langle x,x\rangle, \forall x\in\mathcal{K}.
	\end{equation*}	
	Then $VF$ is a continuous frame for $\mathcal{K}$.
	Moreover,
	$$VSV^{\ast}x=V\int_{\Omega}\langle V^{\ast}x, F_{w}\rangle F_{w} d\mu(w)=\int_{\Omega}\langle x, VF_{w}\rangle VF_{w} d\mu(w), \forall x\in\mathcal{K}.$$
	Then $VSV^{\ast}$ is the frame operator of the continuous frame $VF$.
\end{proof}
\begin{corollary}
	Let $F$ be a continuous frame for $\mathcal{H}$ with respect to $(\Omega,\mu)$ with the frame operator $S$. Then $S^{-\frac{1}{2}}F$ is a Parseval continuous frame for $\mathcal{H}$.
\end{corollary}
\begin{proof}
	Result the next theorem by taking $V=S^{-\frac{1}{2}}$.
\end{proof}

\section{$\ast$-Continuous frame in Hilbert $C^{\ast}$-modules over a unital $C^{\ast}$-algebra}

\begin{definition}
	Let $\mathcal{H}$ be a Hilbert $\mathcal{A}$-module and $(\Omega,\mu)$ a measure space. A mapping $F: \Omega \to \mathcal{H}$ is called a $\ast$-continuous frame with respect to $(\Omega,\mu)$, if
	\begin{itemize}
		\item [1.] for all $x\in\mathcal{H}, w \to \langle x, F_{w}\rangle $ is a measurable function on $\Omega$,
		\item [2.] there exist two strictly nonzero elements $A, B$ in $\mathcal{A}$ such that
		\begin{equation} \label{2..1}
			A\langle x,x\rangle A^{\ast}
			\leq\int_{\Omega}\langle x, F_{w}\rangle\langle F_{w}, x\rangle d\mu(w)\leq B\langle x,x\rangle B^{\ast}, \forall x\in\mathcal{H}.
		\end{equation}
	\end{itemize}
	The elements $A$ and $B$ are called $\ast$-continuous frame bounds. If $A=B$ we call this $\ast$-continuous frame a tight $\ast$-continuous frame, and if $A=B=1_{\mathcal{A}}$ it is called a Parseval $\ast$-continuous frame. If only the right-hand inequality of \eqref{2..1} is satisfied, we call  $F: \Omega \to \mathcal{H}$ a $\ast$-continuous Bessel mapping with $\ast$-Bessel bound $B$.
\end{definition}
\begin{remark}
	The set of all continuous frame can be considered as a subset of $\ast$-continuous frame.
\end{remark}
\begin{example}
	Let $\mathcal{A}$ be the $C^{\ast}$-algebra $\Bigg\{\begin{bmatrix}
		a&0\\0&b
	\end{bmatrix}: a, b\in\mathbb{C}\Bigg\}$, then $\mathcal{A}$ is a Hilbert $C^{\ast}$-module over itself.\\
	Define $F: \Omega \to \mathcal{A}$ by $F_{w}=\begin{bmatrix}
		w&0\\0&w+1
	\end{bmatrix}$.\\
	So 
	\begin{align*}
		\int_{\Omega}\langle T, F_{w}\rangle\langle F_{w}, T\rangle d\lambda&=\int_{0}^{1}TF_{w}^{\ast}F_{w}T^{\ast}d\lambda\\
		&= \begin{bmatrix}
			a&0\\0&b
		\end{bmatrix}\int_{0}^{1}\begin{bmatrix}
			w^{2}&0\\0&(w+1)^{2}
		\end{bmatrix}d\lambda\begin{bmatrix}
			\bar{a}&0\\0&\bar{b}
		\end{bmatrix}\\
		&= \begin{bmatrix}
			a&0\\0&b
		\end{bmatrix}\begin{bmatrix}
			\frac{1}{3}&0\\0&\frac{7}{3}
		\end{bmatrix}\begin{bmatrix}
			\bar{a}&0\\0&\bar{b}
		\end{bmatrix}\\
		&= \begin{bmatrix}
			\frac{1}{\sqrt{3}}&0\\0&\sqrt{\frac{7}{3}}
		\end{bmatrix}\begin{bmatrix}
			|a|^{2}&0\\0&|b|^{2}
		\end{bmatrix}\begin{bmatrix}
			\frac{1}{\sqrt{3}}&0\\0&\sqrt{\frac{7}{3}}
		\end{bmatrix}.
	\end{align*}
	Hence 
	$$\begin{bmatrix}
		\frac{1}{\sqrt{3}}&0\\0&\frac{1}{\sqrt{3}}
	\end{bmatrix}\langle T, T\rangle\begin{bmatrix}
		\frac{1}{\sqrt{3}}&0\\0&\frac{1}{\sqrt{3}}
	\end{bmatrix}\leq\int_{\Omega}\langle T, F_{w}\rangle\langle F_{w}, T\rangle d\lambda\leq\begin{bmatrix}
		\sqrt{\frac{7}{3}}&0\\0&\sqrt{\frac{7}{3}}
	\end{bmatrix}\langle T, T\rangle\begin{bmatrix}
		\sqrt{\frac{7}{3}}&0\\0&\sqrt{\frac{7}{3}}
	\end{bmatrix}.$$
	Then $F$ is a $\ast$-continuous frame for $\mathcal{A}$ with respect to $([0, 1], d\lambda)$, with bounds $\begin{bmatrix}
		\frac{1}{\sqrt{3}}&0\\0&\frac{1}{\sqrt{3}}
	\end{bmatrix}$ and $\begin{bmatrix}
		\sqrt{\frac{7}{3}}&0\\0&\sqrt{\frac{7}{3}}
	\end{bmatrix}$.
\end{example}

Suppose that $F$ is a $\ast$-continuous frame for $\mathcal{H}$ with respect to $(\Omega,\mu)$. The operator $T_{F}: \mathcal{H} \to L^{2}(\Omega, \mathcal{A})$ defined by $T_{F}x(\omega)=\langle x, F_{\omega}\rangle, \omega\in\Omega$, is called the analysis operator. $T_{F}$ is adjointable and $T_{F}^{\ast}: L^{2}(\Omega, \mathcal{A}) \to \mathcal{H}$ is given by $T_{F}^{\ast}\varphi=\int_{\Omega}\varphi(\omega)F_{\omega} d\mu(\omega)$, is called the synthesis operator.
\begin{definition}
	Let $F$ be a $\ast$-continuous frame for $\mathcal{H}$ with respect to $(\Omega,\mu)$. We define the frame operator $S:\mathcal{H}\to\mathcal{H}$ by $Sx=T_{F}^{\ast}T_{F}x=\int_{\Omega}\langle x, F_{w}\rangle F_{w} d\mu(w), \forall x\in\mathcal{H}$.
\end{definition}
\begin{theorem}
	The  $\ast$-continuous frame operator $S$ is bounded, positive, self-adjoint and invertible.
\end{theorem}

\begin{proof}
	Let $x, y\in\mathcal{H}$, we have
	\begin{align*}
		\langle Sx, y\rangle&=\Big\langle\int_{\Omega}\langle x, F_{w}\rangle F_{w} d\mu(w), y\Big\rangle\\&=\int_{\Omega}\langle x, F_{w}\rangle\langle F_{w}, y\rangle d\mu(w)\\&=\Big\langle x, \int_{\Omega}\langle y, F_{w}\rangle F_{w}d\mu(w)\Big\rangle\\&=\langle x, Sy\rangle,
	\end{align*}
	so the operator $S$ is self-adjoint.
	
	\noindent
	Let $x\in\mathcal{H}$, by the definition of a $\ast$-continuous frame for $\mathcal{H}$ we have
	\begin{equation*}
		A\langle x,x\rangle A^{\ast}
		\leq\int_{\Omega}\langle x, F_{w}\rangle\langle F_{w}, x\rangle d\mu(w)\leq B\langle x,x\rangle B^{\ast}.
	\end{equation*}
	So
	\begin{equation}\label{25.}
		A\langle x,x\rangle A^{\ast}\leq\langle Sx,x\rangle\leq B\langle x,x\rangle B^{\ast}.
	\end{equation}
	Thus $S$ is positive, and by inequality \eqref{25.} and Theorem $2.5$ in \cite{Moosavi-Nazari-2019-IntJAnalAppl} $S$ is invertible.
\end{proof}
\begin{lemma}
	$ F_{w }$ is a ${\ast}$-continuous Bessel family for Hilbert $C^{\ast}$-module $\mathcal{H}$ with respect to $(\Omega,\mu)$, if and only if there exists constant $\alpha\geq 0$, such that $S\geq \alpha\alpha^{\ast}$, where $S$ is the ${\ast}$-frame operator of $ F_{w }$.
\end{lemma}
\begin{proof}
	The family $F_{w }$ is a ${\ast}$-continuous  Bessel for Hilbert $C^{\ast}$-module $\mathcal{H}$ with bound  $B$ if and only if 
	$$ \int_{\Omega}\langle x,F_{w}x\rangle\langle F_{w}x,x\rangle d\mu(w)\leq B\langle x,x\rangle B^{\ast} , \forall x\in \mathcal{H}.$$
	That is 
	$$ \int_{\Omega}\langle \langle x,F_{w}x\rangle F_{w} x,x\rangle d\mu(w)\leq B\langle x,x\rangle B^{\ast} , \forall x\in \mathcal{H},$$
	so 
	$$\langle \int_{\Omega} \langle x,F_{w}x\rangle F_{w} x,x\rangle  \leq B\langle x,x\rangle B^{\ast} , \forall x\in \mathcal{H}.$$
	so 
	$$\alpha\langle x,x\rangle \alpha^{\ast} \leq\langle Sx,x\rangle \leq B\langle x,x\rangle B^{\ast} , \forall x\in \mathcal{H},$$
	where $S$ is the ${\ast}$-continuous frame operator of $(F_{w})_{w\in \Omega }$.\\
	Therefore, the conclusion holds.
\end{proof}
\begin{theorem}
	Let  $ F_{w }$ be a ${\ast}$-continuous frame for Hilbert $C^{\ast}$-module $\mathcal{H}$ with ${\ast}$-continuous frame operator $S$ and
	lower and upper ${\ast}$-continuous frame bounds $\sqrt{A}$ and $\sqrt{B}$, respectively. Suppose that is a strictly
	positive element in $\mathcal{A}$. Then the sequence $\{\alpha F_{w }:w\in \Omega\}$ is a ${\ast}$-continuous frame for $\mathcal{H}$ with ${\ast}$-continuous frame operator
	${\mid\alpha\mid}^{2}S$
	and lower and upper ${\ast}$-continuous frame bounds $\alpha)\sqrt{A}$ and $\alpha\sqrt{B}$, respectively.
\end{theorem}
\begin{proof}
	For $x\in \mathcal{H}$, we have $$\int_{\Omega}\langle x,\alpha F_{w}\rangle\langle \alpha F_{w}, x\rangle d\mu(w) = \mid\alpha\mid\int_{\Omega}\langle x,F_{w}\rangle\langle F_{w}, x\rangle d\mu(w), \qquad x\in \mathcal{H}.$$
	Therefore $\alpha F_{w}:w\in \Omega$ is a ${\ast}$-continuous frame for $\mathcal{H}$ with lower and upper ${\ast}$-continuous frame bounds $\alpha\sqrt{A}$ and $\alpha\sqrt{B}$, respectively. If $S_{\alpha}$ is ${\ast}$-continuous frame operator then $$S_{\alpha}x=\int_{\Omega}\langle x,\alpha F_{w}\rangle\alpha F_{w} d\mu(w)=\alpha\int_{\Omega}\langle x, F_{w}\rangle F_{w} d\mu(w)=\alpha S_{\alpha}x, \qquad x\in \mathcal{H}.  $$
\end{proof}
\begin{theorem}
	Let  $ F_{w }$ be a ${\ast}$-continuous frame for Hilbert $C^{\ast}$-module $\mathcal{H}$ with ${\ast}$-continuous frame operator $S$ and lower and upper ${\ast}$-continuous frame bounds $ A$ and $B$, respectively.  in the center of $\mathcal{A}$. Suppose that $ f$ is an element in  $\mathcal{H}$ such that$ \langle F,F\rangle $ is an invertible element in the center of $\mathcal{A}$. 
	Then the sequence $\{ \langle F_{w },F\rangle: w\in \Omega \} $ is a ${\ast}$-continuous frame for Hilbert $C^{\ast}$-module $\mathcal{H}$ A with lower and upper ${\ast}$-continuous frame for Hilbert $C^{\ast}$-module $\mathcal{H}$ frame bounds $ A\sqrt{\langle F,F\rangle}$ and $ B\sqrt{\langle F,F\rangle}$, respectively. And its ${\ast}$-continuous operator is $S_{F }a =\langle SF,F\rangle$ for $ a\in \mathcal{A}$.
\end{theorem}
\begin{proof}
	For $a \in\mathcal{A} $, by the definition of ${\ast}$-continuous frame
	\begin{equation} \label{2..1}
		aA\langle x,x\rangle A^{\ast}a^{\ast}
		\leq a\int_{\Omega}\langle x, F_{w}\rangle\langle F_{w}, x\rangle d\mu(w) a^{\ast}\leq aB\langle x,x\rangle B^{\ast}a^{\ast}, \forall x\in\mathcal{H},
	\end{equation}
	and we have $$\int_{\Omega}\langle  x,\langle F_{w},x\rangle \rangle \langle \langle F_{w}, x \rangle,a \rangle d\mu(w)= a\int_{\Omega}\langle x, F_{w}\rangle\langle F_{w}, x\rangle d\mu(w)a^{\ast}.$$
\end{proof}
\begin{theorem}
	Let  $ F_{w }$ be a ${\ast}$-continuous frame for Hilbert $C^{\ast}$-module $\mathcal{H}$ with ${\ast}$-continuous frame operator $S$ and
	lower and upper ${\ast}$-continuous frame bounds $ A$ and $B$, respectively. Then $S$ is positive, invertible and
	adjointable. Also, the following inequality $\lVert A^{-1}\lVert^{-2} \langle x,x\rangle \leq\langle Sx,x\rangle\leq \lVert B\lVert^{-2} \langle x,x\rangle B^{\ast}$ holds, and the reconstruction formula $F= \int_{\Omega}\langle F,S^{-1}F_{w}\rangle F_{w} d\mu(w) , \forall F\in \mathcal{H} $. Moreover $\{F_{w}:w \in \Omega\}$ is a set of
	module generators of $\mathcal{H} $.
\end{theorem}
\begin{proof}
	The definition of ${\ast}$-continuous frames concludes that
	$$\langle x,x\rangle\leq A^{-1}\langle Sx,x\rangle (A^{\ast})^{-1} and \langle Sx,x\rangle\leq B^{-1}\langle Sx,x\rangle (B^{\ast})^{-1} $$	
	and then $$\lVert A^{-1}\lVert^{-2}\lVert x\lVert^{-2}\leq\lVert\langle Sx,x\rangle\lVert\leq \lVert B\lVert^{2}\lVert x\lVert^{2}.$$
	If we take supremum on all $f \in \mathcal{H}$, where $\lVert f\lVert \leq 1$, then $ \lVert A^{-1}\lVert^{-2}\leq\lVert S\lVert\leq \lVert B\lVert^{2}$.
	The reconstruction formula concludes by the invertibility of $ S $ similar to ordinary frames.	
\end{proof}

\begin{theorem}
	Let  $ \{F_{w} w\in\Omega\}$ be a $\ast$-continuous frames for $\mathcal{H}$ with respect to $(\Omega,\mu)$. with bounds $A$, $B$. Let $T \in End_{\mathcal{A}}^{\ast}(\mathcal{H})$ be invertible  then $\{F_{w}T\}_{\omega \in \Omega}$ is a $\ast$-continuous frames. 
\end{theorem}
\begin{proof}
	We have for all $x\in \mathcal{H}$, $Tx\in \mathcal{H}$,
	\begin{align*}
		A\langle Tx,Tx\rangle A^{\ast} \leq \int_{\Omega}\langle Tx,F_{w}Tx\rangle\langle F_{w}Tx,Tx\rangle d\mu(w) &\leq B\langle Tx,Tx\rangle B^{\ast}\\
		&\leq B\|T\|^{2}\langle x,x\rangle B^{\ast}\\
		&\leq (B\|T\|)\langle x,x\rangle (B\|T\|)^{\ast}.
	\end{align*}
	On other hand, $T$ is invertible then, there exist $0\leq m$ such that
	$$m\langle x,x\rangle m^{\ast}\leq \langle Tx,Tx\rangle.  $$
	So,
	\begin{equation*}
		(Am)\langle x,x\rangle (Am)^{\ast}\leq A\langle Tx,Tx\rangle A^{\ast}
	\end{equation*}
	then 
	\begin{equation*}
		(Am)\langle x,x\rangle (Am)^{\ast}\leq \int_{\Omega}TT^{\ast}\langle x,F_{w}Tx\rangle\langle F_{w} Tx,x\rangle d\mu(w) \leq (B\|T\|)\langle x,x\rangle (B\|T\|)^{\ast}
	\end{equation*}
	this show that $\{F_{w}T\}_{\omega \in \Omega}$ is a $\ast$-continuous frames.
\end{proof}
\begin{theorem}
	Let $F$ be a $\ast$-continuous frame for $\mathcal{H}$ with respect to $(\Omega,\mu)$ with the $\ast$-frame operator $S$. Let $V\in End_{\mathcal{A}}^{\ast}(\mathcal{H},\mathcal{K})$ be a surjective operator. Then $VF$ is a $\ast$-continuous frame for $\mathcal{K}$ with the $\ast$-frame operator $VSV^{\ast}$.
\end{theorem}

\begin{proof}
	The mapping $VF: \Omega \to \mathcal{K}$ is measurable.
	
	\noindent
	Therefore,
	\begin{equation*}
		A\langle V^{\ast}x,V^{\ast}x\rangle A^{\ast}
		\leq\int_{\Omega}\langle V^{\ast}x, F_{w}\rangle\langle F_{w}, V^{\ast}x\rangle d\mu(w)\leq B\langle V^{\ast}x,V^{\ast}x\rangle B^{\ast}, \forall x\in\mathcal{K}.
	\end{equation*}
	So
	\begin{equation*}
		A\|(VV^{\ast})^{-1}\|^{-1}\langle x,x\rangle A^{\ast}
		\leq\int_{\Omega}\langle x, VF_{w}\rangle\langle VF_{w}, x\rangle d\mu(w)\leq B\|V\|^{2}\langle x,x\rangle B^{\ast}, \forall x\in\mathcal{K}.
	\end{equation*}	
	Then $VF$ is a $\ast$-continuous frame for $\mathcal{K}$.
	
	\noindent
	Moreover,
	$$VSV^{\ast}x=V\int_{\Omega}\langle V^{\ast}x, F_{w}\rangle F_{w} d\mu(w)=\int_{\Omega}\langle x, VF_{w}\rangle VF_{w} d\mu(w), \forall x\in\mathcal{K}.$$
	Then $VSV^{\ast}$ is the $\ast$-frame operator of the $\ast$-continuous frame $VF$.	
\end{proof}
\begin{corollary}
	Let $F$ be a $\ast$-continuous frame for $\mathcal{H}$ with respect to $(\Omega,\mu)$ with the frame operator $S$. Then $S^{-\frac{1}{2}}F$ is a Parseval $\ast$-continuous frame for $\mathcal{H}$.
\end{corollary}
\begin{proof}
	Result the next theorem by taking $V=S^{-\frac{1}{2}}$.
\end{proof}
In the following we study $\ast$-continuous  frames in two Hilbert $C^{\ast}$-modules with different $C^{\ast}$-algebras.
\begin{theorem}
	Let $ F_{w} \in \mathcal{H}$ be a $\ast$-continuous frame for $ \mathcal{H}$ with lower and upper $\ast$-frame bounds
	$A$ and $ B$, respectively. The $\ast$-frame transform or pre -$\ast$- frame operator $T_{F}: \mathcal{H} \to L^{2}(\Omega, \mathcal{A})$
	defined by $\{ T(F_{w}) = \langle x,F_{w}\rangle; w \in \Omega\}$ is an injective and closed range adjointable $ A $-module map
	and $\lVert T\lVert \leq \lVert B\lVert$.   The adjoint operator $ T^{\ast}$ is surjective and it is given by $T^{\ast}(e_{w})=F_{w}$  for $ w \in \Omega$  where $\{e_{w}: w \in \Omega\}$ is the standard basis for $ L^{2}(\Omega, \mathcal{A})$.
\end{theorem}
\begin{proof}
	By the definition of norm in $ L^{2}(\Omega, \mathcal{A})$,
	$$\lVert TF_{w}\lVert^{2} =\int_{\Omega}\langle x,F_{w}x\rangle\langle F_{w}x,x\rangle d\mu(w)\leq B\langle x,x\rangle B^{\ast} , \forall x\in \mathcal{H}.$$
	This inequality implies that $ T$ is well defined and $\lVert T\lVert \leq \lVert B\lVert$. Clearly, $T$ is a linear A-
	module map. We now show that $R_{T}$ is closed. Let $\{ TF_{n}: n \in\mathbb{N} \} $  be a sequence in $R_{T}$ such that
	$ TF_{n}\rightarrow g $ as $ n\rightarrow \infty $ \\  we have $$A\langle F_{n}-F_{m},F_{n}-F_{m}\rangle A^{\ast}\leq \lVert F_{n}-F_{m}\lVert^{2}.$$   
	Since, $\{ TF_{n}: n \in\mathbb{N} \} $
	is a cauchy sequence in $ L^{2}(\Omega, \mathcal{A})$,
	\\ $\lVert A\langle F_{n}-F_{m},F_{n}-F_{m}\rangle A^{\ast}\lVert \rightarrow 0$.
	Note as $ n,m\rightarrow \infty $,
	that for $n,m \in \mathbb{N}$,
	$$\lVert\langle F_{n}-F_{m},F_{n}-F_{m}\rangle\lVert\leq\lVert AA^{-1}\langle F_{n}-F_{m},F_{n}-F_{m}\rangle A^{\ast}(A^{\ast})^{-1}\lVert \leq\lVert A^{-1}\lVert^{2}\lVert A( F_{n}-F_{m})A^{\ast}\lVert^{2}.$$ 
	Therefore the sequence $\{ F_{n}: n \in\mathbb{N} \} $ is a cauchy sequence in $\mathcal{H}$  and hence there exists $ F\in\mathcal{H} $ 
	such that $ TF_{n}\rightarrow g $ as $ n\rightarrow \infty $.\\ Again by the definition of  $\ast$-continuous frame for $ \mathcal{H}$, we obtain $$\lVert T( F_{n}-F) \lVert \leq\lVert B\lVert^{2}\lVert ( F_{n}-F)\lVert$$  
	Thus $\lVert T( F_{n}-F)\rightarrow 0$  as $ n\rightarrow \infty$  implies that $TF=g$. It concludes
	that $R_{T}$  is closed. In order to show that $ T$ is injective, suppose that $F \in\mathcal{H}$  and $TF=0$.  
	$$\lVert\langle F,F\rangle\lVert\leq\lVert AA^{-1}\langle F,F\rangle A^{\ast}(A^{\ast})^{-1}\lVert \leq\lVert A^{-1}\lVert^{2}\lVert TF\lVert^{2}.$$ 
	Thus $F=0$  and $T$ is injective. To determine the adjoint operator $T^{\ast}$, consider the equalities
	$$ \langle TF,e_{k}\rangle =\langle \{ \langle F,F_{w}\rangle\}_{w},e_{k}\rangle =\langle F,F_{k}\rangle$$
	for all $k\in \Omega$  and  $ F\in\mathcal{H} $.\\
	Now, given  $ F\in\mathcal{H} $  and
	$\{a_{w}\in L^{2}(\Omega, \mathcal{A}):w \in \Omega\} $, we have
	\begin{align*}
		\langle a_{w}, TF\rangle&=\int_{\Omega} a_{w}\langle F , F_{w}\rangle^{\ast}  d\mu(w)\\&=\langle \int_{\Omega}a_{w}F_{w}d\mu(w), F\rangle.  
	\end{align*}
	This implies that $\int_{\Omega}a_{w}F_{w}d\mu(w)$
	converges in $ \mathcal{H}$ and 
	for every $\{a_{w}\in L^{2}(\Omega, \mathcal{A}):w \in \Omega\} $. 
	By injectivity of $T$, the operator $T^{\ast}$  has closed range and $R_{T}=\mathcal{H}$, which completes the proof. 
\end{proof}
\begin{theorem}
	Let $(\mathcal{H},\mathcal{A}, \langle.  ,.\rangle_{\mathcal{A}})$ and $(\mathcal{H},\mathcal{B}, \langle.  ,.\rangle_{\mathcal{B}})$ be two Hilbert $C^{\ast}$-modules,  $\phi :\mathcal{A} \rightarrow \mathcal{B}$ be a $\ast$-homomorphism and $\theta$ be an adjointable map on $\mathcal{H}$ such that $\langle \theta x,\theta y\rangle_{\mathcal{B}}=\phi(\langle x,y\rangle_{\mathcal{A}})$ for all $x, y\in \mathcal{H}$. Also, suppose that $\{F_{w}\}_{w\in\Omega}$ is a ${\ast}$-continuous frame for $(\mathcal{H}, \mathcal{A},\langle.  ,.\rangle_{\mathcal{A}})$ with  $\ast$-continuous frame operator $S_{\mathcal{A}}$ and lower and upper bounds $A$, $B$ respectively. If $\theta$ is surjective and $\theta F_{w}=F_{w}\theta$ for all $w\in\Omega$, then $\{F_{w}\}_{w\in\Omega}$ is a $\ast$-continuous  frame for $(\mathcal{H},\mathcal{B}, \langle.  ,.\rangle_{\mathcal{B}})$ with $\ast$-continuous frame operator $S_{\mathcal{B}}$ and lower and upper bounds $\phi(A)$ and $\phi(B)$, respectively, and $\langle S_{\mathcal{B}}\theta x,\theta y\rangle_{\mathcal{B}} =\phi(\langle S_{\mathcal{A}}x, y\rangle_{\mathcal{A}}).$ 
\end{theorem}

\begin{proof}
	Let $y\in \mathcal{H}$. Since $\theta$ is surjective,  there exists $x\in \mathcal{H}$ such that $\theta x=y$, and $\{F_{w}\}_{w\in\Omega}$ is a ${\ast}$-continuous frame for $\mathcal{H}$ we have 
	\begin{equation*}
		A\langle x,x\rangle_{\mathcal{A}} A^{\ast}\leq \int_{\Omega}\langle x,F_{w}x\rangle_{\mathcal{A}}\langle F_{w}x,x\rangle_{\mathcal{A}} d\mu(w)\leq B\langle x,x\rangle_{\mathcal{A}} B^{\ast}.
	\end{equation*}
	Thus 
	\begin{equation*}
		\phi(	A\langle x,x\rangle_{\mathcal{A}} A^{\ast})\leq \phi\big( \int_{\Omega}\langle x,F_{w}x\rangle_{\mathcal{A}}\langle F_{w}x,x\rangle_{\mathcal{A}} d\mu(w)\big)\leq \phi(B\langle x,x\rangle_{\mathcal{A}} B^{\ast}).
	\end{equation*}
	By definition of $\ast$-homomorphism, we have 
	\begin{equation*}
		\phi(A)\phi(\langle x,x\rangle_{\mathcal{A}} )\phi(A^{\ast})\leq \int_{\Omega}\phi\big(\langle x,F_{w}x\rangle_{\mathcal{A}}\langle F_{w}x,x\rangle_{\mathcal{A}}\big)d\mu(w)\leq\phi(B)\phi(\langle x,x\rangle_{\mathcal{A}})\phi(B^{\ast}).
	\end{equation*}
	By the relation betwen $\theta$ and $\phi$, we get 
	\begin{equation*}
		\phi(A)\langle y,y\rangle_{\mathcal{B}}\phi(A)^{\ast}\leq \int_{\Omega}\langle y,F_{w}y\rangle_{\mathcal{B}}\langle F_{w}y,y\rangle_{\mathcal{B}}d\mu(w)\leq  \phi(B)\langle y,y\rangle_{\mathcal{B}}\phi(B)^{\ast}.
	\end{equation*}
	On the other hand, we have
	\begin{align*}
		\phi(\langle S_{\mathcal{A}}x, y\rangle_{\mathcal{A}})&=\phi(\langle \int_{\Omega}\langle x,F_{w}x \rangle F_{w}xd\mu(w),y\rangle_{\mathcal{A}})\\&=\int_{\Omega}\phi(\langle x,F_{w}x\rangle_{\mathcal{A}}\langle F_{w}x,y\rangle_{\mathcal{A}})d\mu(w)\\&=\int_{\Omega}\langle  \theta x,F_{w}\theta x\rangle_{\mathcal{B}}\langle  F_{w}\theta x, \theta y\rangle_{\mathcal{B}}d\mu(w)\\&=\langle \int_{\Omega}\langle \theta x,F_{w}\theta x \rangle F_{w}\theta x,\theta y\rangle_{\mathcal{B}} d\mu(w), \theta \\&=\langle S_{\mathcal{B}}\theta x,\theta y\rangle_{\mathcal{B}}.
	\end{align*}
	This completes the proof.	
\end{proof}
\begin{theorem}	\label{Thorem3-1}
	Let $\left(\mathcal{H}, \mathcal{A},\langle\cdot, \cdot\rangle_{\mathcal{A}}\right)$ and $\left(\mathcal{H}, \mathcal{B},\langle\cdot, \cdot\rangle_{\mathcal{B}}\right)$ be two Hilbert $C^{*}$-modules and let $\varphi$ : $\mathcal{A} \longrightarrow \mathcal{B}$ be a $*$-homomorphism and $\theta$ be a map on $\mathcal{H}$ such that $\langle\theta x, \theta y\rangle_{\mathcal{B}}=\varphi\left(\langle x, y\rangle_{\mathcal{A}}\right)$ for all $x, y \in \mathcal{H}$. Also, suppose that $F_{w} $ is a ${\ast}$-continuous frame for $\left(\mathcal{H}, \mathcal{A},\langle\cdot, \cdot\rangle_{\mathcal{A}}\right)$ with ${\ast}$-continuous frame operator $S_{\mathcal{A}}$ and lower and upper $\ast$-continuous frame bounds $\alpha_{1}, \alpha_{2}$, respectively. If $\theta$ is surjective, then $\left\{\theta F_{w}\right\}_{w  \in  \Omega}$ is a ${\ast}$-continuous frame for $\left(\mathcal{H}, \mathcal{B},\langle\cdot, \cdot\rangle_{\mathcal{B}}\right)$ with ${\ast}$-continuous frame operator $S_{\mathcal{B}}$ and lower and upper continuous frame bounds $\varphi\left(\alpha_{1}\right), \varphi\left(\alpha_{2}\right)$, respectively, and
	\begin{equation}\label{eq-3.2}
		\left\langle S_{\mathcal{B}} \theta x, \theta y\right\rangle_{\mathcal{B}}=\varphi\left(\left\langle S_{\mathcal{A}} x, y\right\rangle_{\mathcal{A}}\right), \quad \forall x \in \mathcal{H}.
	\end{equation}
	Moreover, the map $\theta$ is surjective if the following conditions are valid.
	\begin{enumerate}
		\item[(1)] $\varphi$ is surjective;
		\item[(2)] $\left\{\theta F_{w}\right\}_{w \in \Omega }$ is a ${\ast}$-continuous frame for $\mathcal{H}$;
		\item[(3)] $\theta(a x)=\varphi(a) \theta x$, for all $a \in \mathcal{A}, x \in \mathcal{H}$.
	\end{enumerate}
\end{theorem}
\begin{proof}
	Assume that $\theta$ is surjective. Using Proposition \ref{Proposition-0-1}, we have that
	$$
	\begin{aligned}
		\int_{\Omega}\left\langle\theta x, \theta F_{w}\right\rangle_{\mathcal{B}}\left\langle\theta F_{w}, \theta x\right\rangle_{\mathcal{B}}d\mu(w) &=\int_{\Omega} \varphi\left(\left\langle x, F_{w}\right\rangle_{\mathcal{A}}\left\langle F_{w}, x\right\rangle_{\mathcal{A}}\right)d\mu(w) \\
		& \leq \varphi\left(\alpha_{2}\langle x, x\rangle_{\mathcal{A}} \alpha_{2}^{*}\right)=\varphi\left(\alpha_{2}\right)\langle\theta x, \theta x\rangle_{\mathcal{B}} \varphi\left(\alpha_{2}\right)^{*}, \quad \forall x \in \mathcal{H},
	\end{aligned}
	$$
	and $\varphi\left(\alpha_{2}\right)$ is a strictly nonzero element of $\mathcal{B}$. Then the sequence $\left\{\theta F_{w}\right\}_{w \in \Omega}$ has upper ${\ast}$-continuous frame bound $\varphi\left(\alpha_{2}\right)$. Similarly, $\varphi\left(\alpha_{1}\right)$ is a lower ${\ast}$-continuous frame bound for $\left\{\theta F_{w}\right\}_{w \in \Omega}$ and then $\left\{\theta F_{w}\right\}_{w \in \Omega}$ is a ${\ast}$-continuous frame for $\left(\mathcal{H}, \mathcal{B},\langle\cdot, \cdot\rangle_{\mathcal{B}}\right)$. The equation \eqref{eq-3.2} follows from
	$$
	\int_{\Omega}\left\langle\theta x, \theta F_{w}\right\rangle_{\mathcal{B}}\left\langle\theta F_{w}, \theta y\right\rangle_{\mathcal{B}}d\mu(w)=\varphi\left(\int_{\Omega}\left\langle x, F_{w}\right\rangle_{\mathcal{A}}\left\langle F_{w}, y\right\rangle_{\mathcal{A}}\right)d\mu(w), \quad \forall x, y \in \mathcal{H}.  
	$$
	For the rest of the proof, let $\varphi$ be surjective and $\theta(a x)=\varphi(a) \theta x$, for all $a \in \mathcal{A}$ and $x \in \mathcal{H}$. By applying the reconstruction formula for ${\ast}$-continuous frame $\left\{\theta F_{w}\right\}_{w \in \omega}$, we have $y=\int_{\Omega}\left\langle y, S_{\mathcal{B}}^{-1} \theta F_{w}\right\rangle_{\mathcal{B}} \theta F_{w}d\mu(w)$ for $y \in \mathcal{H}$. Since $\varphi$ is surjective, $\varphi\left(a_{w}\right)=\left\langle y, S_{\mathcal{B}}^{-1} \theta F_{w}\right\rangle_{\mathcal{B}}$ for some $ a_{w} \in \mathcal{A}$ and for all $w \in \omega$. Observe that $y=\int_{\Omega} \varphi\left(a_{w}\right) \theta F_{w}d\mu(w)=\int_{\Omega} \theta\left(a_{w} F_{w}\right)d\mu(w)=\theta\left(\int_{\Omega} a_{w} F_{w}d\mu(w)\right)$. This shows that $\theta$ is surjective and the proof is complete.
\end{proof}
\begin{corollary}
	Let $\mathcal{A}, \mathcal{B}, \mathcal{H},\left\{F_{w}\right\}_{w \in \Omega}$ and $\varphi$ be as in Theorem \ref{Thorem3-1}. Also, let $\theta$ be a $\mathcal{B}$-module map on $\mathcal{H}$ such that $\varphi\left(\langle x, y\rangle_{\mathcal{A}}\right)=\langle\theta x, \theta y\rangle_{\mathcal{B}}$. Then $\theta$ is surjective if and only if $\left\{\theta F_{w}\right\}_{w \in \Omega}$ is a $ *$- continuous frame for $\left(\mathcal{H}, \mathcal{B},\langle\cdot, \cdot\rangle_{\mathcal{B}}\right)$.
\end{corollary}
\begin{proof}
	Proof of the 'if part' is similar to the proof of Theorem \ref{Thorem3-1}. For the converse, since $\theta$ is $\mathcal{B}$-module map, $g=\int_{\Omega}\left\langle g, S_{\mathcal{B}}^{-1} \theta F_{w}\right\rangle_{\mathcal{B}} \theta F_{w}d\mu(w)=\theta\left(\int_{\Omega}\left\langle g, S_{\mathcal{B}}^{-1} \theta F_{w}\right\rangle_{\mathcal{B}} F_{w}d\mu(w)\right)$, for $g \in \mathcal{H}$, and it completes the proof.
\end{proof}
\begin{proposition}	
	Let $\mathcal{A}, \mathcal{B}$ and $\mathcal{H}$ be the same in Theorem \ref{Thorem3-1}. If $\varphi$ is a $*$-isomorphism and $\theta$ is surjective map on $\mathcal{H}$ such that $\varphi\left(\langle f, g\rangle_{\mathcal{A}}\right)=\langle\theta f, \theta g\rangle_{\mathcal{B}}$, then the set of all ${\ast}$-continuous frames for $\left(\mathcal{H}, \mathcal{B},\langle\cdot, \cdot\rangle_{\mathcal{B}}\right)$ is precisely $\left\{\theta F_{w}\right\}_{w \in \Omega}$ where $\left\{F_{w}\right\}_{w \in \Omega}$ is a ${\ast}$-continuous frame for $\left(\mathcal{H}, \mathcal{A},\langle\cdot, \cdot\rangle_{\mathcal{A}}\right)$.
\end{proposition}
\begin{proof}
	Theorem \ref{Thorem3-1} concludes that the sequence $\left\{\theta F_{w}\right\}_{w \in \Omega}$ is a ${\ast}$-continuous frame for $\left(\mathcal{H}, \mathcal{B},\langle\cdot, \cdot\rangle_{\mathcal{B}}\right)$ if $\left\{F_{w}\right\}_{w \in \Omega}$ is a ${\ast}$-continuous frame for $\left(\mathcal{H}, \mathcal{A},\langle\cdot, \cdot\rangle_{\mathcal{A}}\right)$. Now, assume that $\left\{g_{w}\right\}_{w \in \Omega}$ is a ${\ast}$-continuous frame for $\left(\mathcal{H}, \mathcal{B},\langle\cdot, \cdot\rangle_{\mathcal{B}}\right)$ with lower and upper ${\ast}$-continuous frame bounds $\beta_{1}$ and $\beta_{2}$. By the properties of $\theta$, and Proposition \ref{Proposition-0-1}, there exist the sequence $\left\{F_{w}\right\}_{w \in \Omega}$ in $\mathcal{H}$ and two elements $\alpha_{1}, \alpha_{2}$ in $\mathcal{A}$ such that $g_{w}=\theta F_{w}$ for $w \in \Omega, \varphi\left(\alpha_{1}\right)=\beta_{1}$, and $\varphi\left(\alpha_{2}\right)=\beta_{2}$. The elements $\alpha_{1}$ and $\alpha_{2}$ are strictly nonzero by Proposition \ref{Proposition-0-1}. Using the definition of the ${\ast}$-continuous frame $\left\{g_{w}\right\}_{w \in \Omega}$, we have
	$$
	\begin{aligned}
		\varphi\left(\int_{\Omega}\left\langle f, F_{w}\right\rangle_{\mathcal{A}}\left\langle F_{w}, f\right\rangle_{\mathcal{A}}d\mu(w)\right) &=\int_{\Omega}\left\langle\theta f, \theta F_{w}\right\rangle_{\mathcal{B}}\left\langle\theta F_{w}, \theta f\right\rangle_{\mathcal{B}}d\mu(w) \\
		& \leq \beta_{2}\langle\theta f, \theta f\rangle_{\mathcal{B}} \beta_{2}^{*}=\varphi\left(\alpha_{2}\langle f, f\rangle_{\mathcal{A}} \alpha_{2}^{*}\right), \quad \forall f \in \mathcal{H}.  
	\end{aligned}
	$$
	We apply Proposition \ref{Proposition-0-1} again, $\int_{\Omega}\left\langle f, F_{w}\right\rangle_{\mathcal{A}}\left\langle F_{w}, f\right\rangle_{\mathcal{A}}d\mu(w) \leq \alpha_{2}\langle f, f\rangle_{\mathcal{A}} \alpha_{2}^{*}$, for $f \in \mathcal{H}$. Similarly, $\alpha_{1}$ is a lower $*$-frame bound for $\left\{F_{w}\right\}_{w \in \Omega}$. This shows that every $*$-frame in $\left(\mathcal{H}, \mathcal{B},\langle\cdot, \cdot\rangle_{\mathcal{B}}\right)$ is obtained by the action of $\theta$ on a $*$-continuous frame in $\left(\mathcal{H}, \mathcal{A},\langle\cdot, \cdot\rangle_{\mathcal{A}}\right)$.
\end{proof}
\begin{proposition}
	Let $\varphi: \mathcal{A} \longrightarrow \mathcal{B}$ be a $*$isomorphism. The set of all of ${\ast}$-continuous frames for the Hilbert $\mathcal{B}$-module $\mathcal{B}$ is precisely $\left\{\varphi\left(a_{w}\right)\right\}_{w \in \Omega}$, where $\left\{a_{w}\right\}_{w \in \Omega}$ is a ${\ast}$-continuous frame for the Hilbert $\mathcal{A}$-module $\mathcal{A}$. Moreover, if $S_{\mathcal{A}}$ and $S_{\mathcal{B}}$ are ${\ast}$-continuous frame operators for $\left\{a_{w}\right\}_{w \in \Omega}$ and $\left\{\varphi\left(a_{w}\right)\right\}_{w \in \Omega}$, respectively, then $\varphi o S_{\mathcal{A}}=S_{\mathcal{B}} \circ \varphi$.
\end{proposition}
\begin{proof}
	For a sequence $\left\{a_{w}\right\}_{w \in \Omega}$ in $\mathcal{A}$, we have
	\begin{equation}\label{eq3-3}
		\int_{\Omega}\left\langle\varphi(a), \varphi\left(a_{w}\right)\right\rangle_{\mathcal{B}}\left\langle\varphi\left(a_{w}\right), \varphi(a)\right\rangle_{\mathcal{B}}d\mu(w)=\varphi\left(\int_{\Omega}\left\langle a, a_{w}\right\rangle_{\mathcal{A}}\left\langle a_{w}, a\right\rangle_{\mathcal{A}}d\mu(w)\right), \quad \forall a \in \mathcal{A}.
	\end{equation}
	Proposition \ref{Proposition-0-1} and the above equalities imply that $\left\{\varphi\left(a_{w}\right)\right\}_{w \in \Omega}$ is a ${\ast}$-continuous frame for $\mathcal{B}$ if $\left\{a_{w}\right\}_{w \in \Omega}$ is a ${\ast}$-continuous frame for $\mathcal{A}$. Now, suppose $\left\{b_{w}\right\}_{w \in \Omega}$ is a ${\ast}$-continuous frame for $\mathcal{B}$. Since $\varphi$ is surjective, there exists a sequence $\left\{a_{w}\right\}_{w \in \Omega}$ in $\mathcal{A}$ such that $b_{w}=\varphi\left(a_{w}\right)$ for $w \in \Omega$. Also, applying Proposition \ref{Proposition-0-1}  and \eqref{eq3-3}, we obtain that $\left\{a_{w}\right\}_{w \in \Omega}$ is a ${\ast}$-continuous frame for $\mathcal{A}$. For the rest of the proof, let $S_{\mathcal{A}}$ and $S_{\mathcal{B}}$ be $*$-frame operators for $\left\{a_{w}\right\}_{w \in \Omega}$ and $\left\{\varphi\left(a_{w}\right)\right\}_{w \in \Omega}$, respectively. Then $\varphi S_{\mathcal{A}}(a)=$ $\varphi\left(\int_{\Omega} a a_{w}^{*} a_{w}\right)=S_{\mathcal{B}} \varphi(a)$, for all $a \in \mathcal{A}$, and $\varphi o S_{\mathcal{A}}=S_{\mathcal{B}} o \varphi$.
\end{proof}
\begin{theorem}
	Let $\{F_{w}: w\in\Omega\}$ be a $\ast$-continuous frame for $\mathcal{H}$ with lower and upper bounds $A$ and $B$, respectively. Let $\theta\in End_{\mathcal{A}}^{\ast}(\mathcal{H})$ be injective and have a closed range. Then $\{\theta F_{w}\}_{w\in \Omega}$ is a $\ast$-continuous frame for $\mathcal{H}$.
\end{theorem}

\begin{proof}
	$\{F_{w}: w\in\Omega\}$ be a $\ast$-continuous frame for $\mathcal{H}$.\\	
	We have 
	\begin{equation*}
		A\langle x,x\rangle A^{\ast}\leq \int_{\Omega}\langle x,F_{w}x\rangle\langle F_{w} x,x\rangle d\mu(w)\leq B\langle x,x\rangle B^{\ast},\qquad\forall x\in U.
	\end{equation*}
	Then for each $x\in \mathcal{H}$ 
	\begin{equation}\label{eq13}
		\int_{\Omega}\langle \theta x,\theta F_{w}x\rangle\langle \theta F_{w} x,\theta x\rangle d\mu(w)\leq B\langle\theta x,\theta x\rangle B^{\ast}\leq\|\theta \|^{2}B\langle x,x\rangle B^{\ast}\leq (\|\theta \|B)\langle x,x\rangle (\|\theta\|B)^{\ast}.
	\end{equation}
	By Lemma \ref{3}, we have for each $x\in \mathcal{H}$
	\begin{equation*}
		\|(\theta^{\ast}\theta)^{-1}\|^{-1}\langle x,F_{w}x\rangle \langle F_{w}x,x\rangle\leq \langle \theta x, \theta F_{w}x\rangle\langle \theta F_{w} x, \theta x\rangle
	\end{equation*}
	and $\|\theta^{-1}\|^{-2}\leq \|(\theta^{\ast}\theta)^{-1}\|^{-1}$. Thus  
	\begin{equation}\label{eq14}
		\|\theta^{-1}\|^{-1}A\langle x,x\rangle (\|\theta^{-1}\|^{-1}A)^{\ast}\leq \int_{\Omega}\langle \theta x,\theta F_{w}x\rangle \langle \theta F_{w} x,\theta x\rangle d\mu(w).
	\end{equation}
	From \eqref{eq13} and \eqref{eq14}, we have for each $x\in \mathcal{H}$
	\begin{eqnarray*}
		\|\theta^{-1}\|^{-1}A\langle x,x\rangle (\|\theta^{-1}\|^{-1}A)^{\ast} & \leq & \int_{\Omega}\langle \theta x,\theta  F_{w}x\rangle\langle \theta F_{w} x,\theta  x\rangle d\mu(w)\\ & \leq &  \|\theta \|^{2}B\langle x,x\rangle B^{\ast}\\ &\leq & (\|\theta \|B)\langle x,x\rangle (\|\theta\|B)^{\ast}.
	\end{eqnarray*}
	We conclude that $\{\theta F_{w}\}_{w\in \Omega}$ is a $\ast$-continuous frame for $U$.
\end{proof}

\begin{theorem}\label{18}
	Let $\{F_{w}: w\in\Omega\}$ be a $\ast$-continuous frame for $\mathcal{H}$ with lower and upper bounds $A$ and $B$, respectively, and with $\ast$-continuous frame operator $S$. Let $\theta\in End_{\mathcal{A}}^{\ast}(\mathcal{H})$ be injective and have a closed range. Then $\{F_{w}\theta: w\in\Omega\}$ is a $\ast$-continuous frame for $\mathcal{H}$.
\end{theorem}
\begin{proof}
	We have 
	\begin{equation}\label{eq11}
		A\langle \theta x,\theta x\rangle A^{\ast}\leq\int_{\Omega}\langle\theta x,F_{w}\theta x\rangle\langle F_{w}\theta x,\theta x\rangle d\mu(w)\leq B\langle \theta x,\theta x\rangle B^{\ast}, \forall x\in U.
	\end{equation}
	Using Lemma \ref{3}, we have $\|(\theta^{\ast}\theta)^{-1}\|^{-1}\langle x,x\rangle\leq\langle \theta x,\theta x\rangle$, $\forall x\in U$. That is, $\|\theta^{-1}\|^{-2}\leq\|(\theta^{\ast}\theta)^{-1}\|^{-1}$. This implies
	\begin{equation}\label{eq22} 
		\|\theta^{-1}\|^{-1}A\langle x,x\rangle(\|\theta^{-1}\|^{-1}A)^{\ast}\leq A\langle \theta x,\theta x\rangle A^{\ast}, \forall x\in U.
	\end{equation}
	And we know that $\langle \theta x,\theta x\rangle\leq\|\theta\|^{2}\langle x,x\rangle$, $\forall x\in U$. This implies that
	\begin{equation}\label{eq33}
		B\langle \theta x,\theta x\rangle B^{\ast}\leq\|\theta \|B\langle x,x\rangle(\|\theta\|B)^{\ast}, \forall x\in U.
	\end{equation}
	Using \eqref{eq11}, \eqref{eq22} and  \eqref{eq33},  we have
	\begin{equation}
		\|\theta^{-1}\|^{-1}A\langle x,x\rangle(\|\theta^{-1}\|^{-1}A)^{\ast}\leq\int_{\Omega}\langle\theta x,F_{w}\theta x\rangle\langle F_{w}\theta x,\theta x\rangle d\mu(w)\leq B\|\theta \|\langle x,x\rangle(B\|\theta \|)^{\ast}, \forall x\in U.
	\end{equation}
	So $\{F_{w}\theta: w\in\Omega\}$ is a $\ast$-continuous frame for $\mathcal{H}$.	
\end{proof}
\begin{corollary}\label{01}
	Let $F_{w}$ be a $\ast$-continuous frame for $\mathcal{H}$, with $\ast$-continuous frame operator $S$. Then $F_{w}S^{-1}$ is a $\ast$-continuous frame for $\mathcal{H}$.
\end{corollary}
\begin{proof}
	The proof follows from Theorem \ref{18}  by taking $\theta=S^{-1}$.
\end{proof}
\section{The stability problem}
The question of stability plays an important role in various
fields of applied mathematics. The classical theorem of
the stability of a base is due to Paley and Wiener. It is based on the fact that a bounded operator $ T $ on a Banach space is invertible if $\|I-T\|<1$.
\begin{theorem} [\cite{Wiener-Paley-1934}]
	Let $\{f_{i}\}_{i\in\mathbb{N}}$ be a basis of a Banach space $X$, and $\{g_{i}\}_{i\in\mathbb{N}}$ a sequence of vectors in $X$. If there exists a constant $\lambda\in[0,1)$ such that
	\begin{equation*}
		\Big\|\sum_{i\in\mathbb{N}}c_{i}(f_{i}-g_{i})\Big\|\leq\lambda\Big\|\sum_{i\in\mathbb{N}}c_{i}f_{i}\Big\|
	\end{equation*}
	for all finite sequence  $\{c_{i}\}_{i\in\mathbb{N}}$ of scalars, then $\{g_{i}\}_{i\in\mathbb{N}}$ is also a basis for $X$.
\end{theorem}
\begin{theorem}
	Let $F$ be a continuous frame for $\mathcal{H}$ with respect to $(\Omega, \mu)$. Let $G: \Omega \to \mathcal{H}$ be a mapping such that  for all $x\in\mathcal{H}, w \to \langle x, G_{w}\rangle $ is a measurable function on $\Omega$. Then the following are equivalent:
	\begin{itemize}
		\item [1.] $G$ is a continuous frame for $\mathcal{H}$ with respect to $(\Omega, \mu)$.
		\item[2.] There exists a constant $M>0$, such that for any $x\in \mathcal{H}$, one has
		\begin{multline}\label{3..1}
			\bigg\|\int_{\Omega}\langle x, F_{w}-G_{w}\rangle\langle F_{w}-G_{w}, x\rangle d\mu(w)\bigg\|\\
			\leq M\min\bigg(\bigg\|\int_{\Omega}\langle x, F_{w}\rangle\langle F_{w}, x\rangle d\mu(w)\bigg\|,\bigg\|\int_{\Omega}\langle x, G_{w}\rangle\langle G_{w}, x\rangle d\mu(w)\bigg\|\bigg).
		\end{multline}
	\end{itemize}
\end{theorem}
\begin{proof}
	$1.\Rightarrow2.$ Suppose that $G$ is a Continuous frame for $\mathcal{H}$ with lower and upper bounds $C$ and $D$, respectively. Then for any $x\in \mathcal{H}$, we have
	\begin{multline*}
		\bigg\|\int_{\Omega}\langle x, F_{w}-G_{w}\rangle\langle F_{w}-G_{w}, x\rangle d\mu(w)\bigg\|^{\frac{1}{2}}=\big\|T_{F}x-T_{G}x\big\|\\
		\leq\big\|T_{F}x\big\|+\big\|T_{G}x\big\|\\
		=\bigg\|\int_{\Omega}\langle x, F_{w}\rangle\langle F_{w}, x\rangle d\mu(w)\bigg\|^{\frac{1}{2}}+\bigg\|\int_{\Omega}\langle x, G_{w}\rangle\langle G_{w}, x\rangle d\mu(w)\bigg\|^{\frac{1}{2}}\\
		\leq\sqrt{B}\|\langle x,x\rangle\|^{\frac{1}{2}}+\bigg\|\int_{\Omega}\langle x, G_{w}\rangle\langle G_{w}, x\rangle d\mu(w)\bigg\|^{\frac{1}{2}}\\
		\leq\sqrt{\frac{B}{C}}\bigg\|\int_{\Omega}\langle x, G_{w}\rangle\langle G_{w}, x\rangle d\mu(w)\bigg\|^{\frac{1}{2}}+\bigg\|\int_{\Omega}\langle x, G_{w}\rangle\langle G_{w}, x\rangle d\mu(w)\bigg\|^{\frac{1}{2}}\\
		=\bigg(\sqrt{\frac{B}{C}}+1\bigg)\bigg\|\int_{\Omega}\langle x, G_{w}\rangle\langle G_{w}, x\rangle d\mu(w)\bigg\|^{\frac{1}{2}}.
	\end{multline*}
	Similary, we have
	\begin{align*}
		\bigg\|\int_{\Omega}\langle x, F_{w}-G_{w}\rangle\langle F_{w}-G_{w}, x\rangle d\mu(w)\bigg\|^{\frac{1}{2}}\leq\bigg(\sqrt{\frac{D}{A}}+1\bigg)\bigg\|\int_{\Omega}\langle x, F_{w}\rangle\langle F_{w}, x\rangle d\mu(w)\bigg\|^{\frac{1}{2}}.
	\end{align*}
	Let $M=\min\Bigg\{\bigg(\sqrt{\frac{B}{C}}+1\bigg)^{2},\bigg(\sqrt{\frac{D}{A}}+1\bigg)^{2}\Bigg\}$, then the inequality \eqref{3..1} holds.
	
	\noindent
	$2.\Rightarrow1.$ Suppose that the inequality \eqref{3..1} holds. For any $x\in \mathcal{H}$, we have
	\begin{multline*}
		\sqrt{A}\|\langle x,x\rangle\|^{\frac{1}{2}}\leq\bigg\|\int_{\Omega}\langle x, F_{w}\rangle\langle F_{w}, x\rangle d\mu(w)\bigg\|^{\frac{1}{2}}\\
		\leq\bigg\|\int_{\Omega}\langle x, F_{w}-G_{w}\rangle\langle F_{w}-G_{w}, x\rangle d\mu(w)\bigg\|^{\frac{1}{2}}+\bigg\|\int_{\Omega}\langle x, G_{w}\rangle\langle G_{w}, x\rangle d\mu(w)\bigg\|^{\frac{1}{2}}\\
		\leq M^{\frac{1}{2}}\bigg\|\int_{\Omega}\langle x, G_{w}\rangle\langle G_{w}, x\rangle d\mu(w)\bigg\|^{\frac{1}{2}}+\bigg\|\int_{\Omega}\langle x, G_{w}\rangle\langle G_{w}, x\rangle d\mu(w)\bigg\|^{\frac{1}{2}}\\
		=\big(1+M^{\frac{1}{2}}\big)\bigg\|\int_{\Omega}\langle x, G_{w}\rangle\langle G_{w}, x\rangle d\mu(w)\bigg\|^{\frac{1}{2}}.
	\end{multline*}
	Also we obtain
	\begin{multline*}
		\bigg\|\int_{\Omega}\langle x, G_{w}\rangle\langle G_{w}, x\rangle d\mu(w)\bigg\|^{\frac{1}{2}}\\
		\leq\bigg\|\int_{\Omega}\langle x, F_{w}-G_{w}\rangle\langle F_{w}-G_{w}, x\rangle d\mu(w)\bigg\|^{\frac{1}{2}}+\bigg\|\int_{\Omega}\langle x,F_{w}\rangle\langle F_{w},x\rangle d\mu(w)\bigg\|^{\frac{1}{2}}\\
		\leq M^{\frac{1}{2}}\bigg\|\int_{\Omega}\langle x,F_{w}\rangle\langle F_{w},x\rangle d\mu(w)\bigg\|^{\frac{1}{2}}+\bigg\|\int_{\Omega}\langle x,F_{w}\rangle\langle F_{w},x\rangle d\mu(w)\bigg\|^{\frac{1}{2}}\\
		=\big(1+M^{\frac{1}{2}}\big)\bigg\|\int_{\Omega}\langle x,F_{w}\rangle\langle F_{w},x\rangle d\mu(w)\bigg\|^{\frac{1}{2}}\\
		\leq\big(1+M^{\frac{1}{2}}\big)\sqrt{B}\|\langle x,x\rangle\|^{\frac{1}{2}}.
	\end{multline*}
	So $G$ is a continuous frame for $\mathcal{H}$.
\end{proof}
\begin{theorem}
	Let $F$ be a $\ast$-continuous frame for $\mathcal{H}$ with respect to $(\Omega, \mu)$. Let $G: \Omega \to \mathcal{H}$ be a mapping such that  for all $x\in\mathcal{H}, w \to \langle x, G_{w}\rangle $ is a measurable function on $\Omega$. Then the following are equivalent:
	\begin{itemize}
		\item [1.] $G$ is a $\ast$-continuous frame for $\mathcal{H}$ with respect to $(\Omega, \mu)$.
		\item[2.] There exists a constant $M>0$, such that for any $x\in \mathcal{H}$, one has
		\begin{multline}\label{3.1}
			\bigg\|\int_{\Omega}\langle x, F_{w}-G_{w}\rangle\langle F_{w}-G_{w}, x\rangle d\mu(w)\bigg\|\\
			\leq M\min\bigg(\bigg\|\int_{\Omega}\langle x, F_{w}\rangle\langle F_{w}, x\rangle d\mu(w)\bigg\|,\bigg\|\int_{\Omega}\langle x, G_{w}\rangle\langle G_{w}, x\rangle d\mu(w)\bigg\|\bigg).
		\end{multline}
	\end{itemize}
\end{theorem}
\begin{proof}
	$1.\Rightarrow2.$ Suppose that $G$ is a $\ast$-continuous frame for $\mathcal{H}$ with lower and upper bounds $C$ and $D$, respectively. Then for any $x\in \mathcal{H}$, we have
	\begin{multline*}
		\bigg\|\int_{\Omega}\langle x, F_{w}-G_{w}\rangle\langle F_{w}-G_{w}, x\rangle d\mu(w)\bigg\|^{\frac{1}{2}}=\big\|T_{F}x-T_{G}x\big\|\\
		\leq\big\|T_{F}x\big\|+\big\|T_{G}x\big\|\\
		=\bigg\|\int_{\Omega}\langle x, F_{w}\rangle\langle F_{w}, x\rangle d\mu(w)\bigg\|^{\frac{1}{2}}+\bigg\|\int_{\Omega}\langle x, G_{w}\rangle\langle G_{w}, x\rangle d\mu(w)\bigg\|^{\frac{1}{2}}\\
		\leq\|B\|\|\langle x,x\rangle\|^{\frac{1}{2}}+\bigg\|\int_{\Omega}\langle x, G_{w}\rangle\langle G_{w}, x\rangle d\mu(w)\bigg\|^{\frac{1}{2}}\\
		\leq\|B\|\|C^{-1}\|\bigg\|\int_{\Omega}\langle x, G_{w}\rangle\langle G_{w}, x\rangle d\mu(w)\bigg\|^{\frac{1}{2}}+\bigg\|\int_{\Omega}\langle x, G_{w}\rangle\langle G_{w}, x\rangle d\mu(w)\bigg\|^{\frac{1}{2}}\\
		=\bigg(\|B\|\|C^{-1}\|+1\bigg)\bigg\|\int_{\Omega}\langle x, G_{w}\rangle\langle G_{w}, x\rangle d\mu(w)\bigg\|^{\frac{1}{2}}.
	\end{multline*}
	Similary we have
	\begin{align*}
		\bigg\|\int_{\Omega}\langle x, F_{w}-G_{w}\rangle\langle F_{w}-G_{w}, x\rangle d\mu(w)\bigg\|^{\frac{1}{2}}\leq\bigg(\|D\|\|A^{-1}\|+1\bigg)\bigg\|\int_{\Omega}\langle x, F_{w}\rangle\langle F_{w}, x\rangle d\mu(w)\bigg\|^{\frac{1}{2}}.
	\end{align*}
	Let $M=\min\Bigg\{\bigg(\|B\|\|C^{-1}\|+1\bigg)^{2},\bigg(\|D\|\|A^{-1}\|+1\bigg)^{2}\Bigg\}$, then the inequality \eqref{3.1} holds.

	\noindent
	$2.\Rightarrow1.$ Suppose that the inequality \eqref{3.1} holds. For any $x\in \mathcal{H}$, we have
	\begin{multline*}
		\|A^{-1}\|^{-1}\|\langle x,x\rangle\|^{\frac{1}{2}}\leq\bigg\|\int_{\Omega}\langle x, F_{w}\rangle\langle F_{w}, x\rangle d\mu(w)\bigg\|^{\frac{1}{2}}\\
		\leq\bigg\|\int_{\Omega}\langle x, F_{w}-G_{w}\rangle\langle F_{w}-G_{w}, x\rangle d\mu(w)\bigg\|^{\frac{1}{2}}+\bigg\|\int_{\Omega}\langle x, G_{w}\rangle\langle G_{w}, x\rangle d\mu(w)\bigg\|^{\frac{1}{2}}\\
		\leq M^{\frac{1}{2}}\bigg\|\int_{\Omega}\langle x, G_{w}\rangle\langle G_{w}, x\rangle d\mu(w)\bigg\|^{\frac{1}{2}}+\bigg\|\int_{\Omega}\langle x, G_{w}\rangle\langle G_{w}, x\rangle d\mu(w)\bigg\|^{\frac{1}{2}}\\
		=\big(1+M^{\frac{1}{2}}\big)\bigg\|\int_{\Omega}\langle x, G_{w}\rangle\langle G_{w}, x\rangle d\mu(w)\bigg\|^{\frac{1}{2}}.
	\end{multline*}
	Also we obtain
	\begin{multline*}
		\bigg\|\int_{\Omega}\langle x, G_{w}\rangle\langle G_{w}, x\rangle d\mu(w)\bigg\|^{\frac{1}{2}}\\
		\leq\bigg\|\int_{\Omega}\langle x, F_{w}-G_{w}\rangle\langle F_{w}-G_{w}, x\rangle d\mu(w)\bigg\|^{\frac{1}{2}}+\bigg\|\int_{\Omega}\langle x,F_{w}\rangle\langle F_{w},x\rangle d\mu(w)\bigg\|^{\frac{1}{2}}\\
		\leq M^{\frac{1}{2}}\bigg\|\int_{\Omega}\langle x,F_{w}\rangle\langle F_{w},x\rangle d\mu(w)\bigg\|^{\frac{1}{2}}+\bigg\|\int_{\Omega}\langle x,F_{w}\rangle\langle F_{w},x\rangle d\mu(w)\bigg\|^{\frac{1}{2}}\\
		=\big(1+M^{\frac{1}{2}}\big)\bigg\|\int_{\Omega}\langle x,F_{w}\rangle\langle F_{w},x\rangle d\mu(w)\bigg\|^{\frac{1}{2}}\\
		\leq\big(1+M^{\frac{1}{2}}\big)\|B\|\|\langle x,x\rangle\|^{\frac{1}{2}}.
	\end{multline*}
	So $G$ is a $\ast$-continuous frame for $\mathcal{H}$.
\end{proof}

\begin{theorem}
	Let $\{F_{w}\}_{w\in\Omega}$ be a  $\ast$-continuous frame for $\mathcal{H}$ with respect to $(\Omega, \mu)$, and with bounds $A$ and $B$. If $\{\Gamma_{w}\}_{w\in\Omega}$ is a $\ast$-continuous Bessel sequence with bound $E$ such that $\|A^{-1}\|^{-1}\geq \|E\|$, then $\{\Gamma_{w}+F_{w}\}_{w\in\Omega}$ is a $\ast$-continuous frame for $\mathcal{H}$ with respect to $(\Omega, \mu)$.
\end{theorem}
\begin{proof}
	Let $x\in \mathcal{H}$. Then  we have 
	\begin{align*}
		\left\|\int_{\Omega}\langle x,(F_{w}+\Gamma_{w})x\rangle\langle (F_{w}+\Gamma_{w})x,x\rangle d\mu(w)\right\|^{\frac{1}{2}}&=\|\{(F_{w}+\Gamma_{w})\}_{w\in\Omega}\|\\
		&\leq \|\{F_{w}x\}_{w\in\Omega}\|+\|\{\Gamma_{w}x\}_{w\in\Omega}\|\\
		&\leq \left\|\int_{\Omega}\langle x,F_{w}x\rangle\langle F_{w}x,x\rangle d\mu(w)\right\|^{\frac{1}{2}}+\\
		& \qquad	\left\|\int_{\Omega}\langle x,\Gamma_{w}x\rangle\langle \Gamma_{w}x,x\rangle d\mu(w)\right\|^{\frac{1}{2}} \\
		&\leq \|B\langle x,x\rangle B^{\ast} \|^{\frac{1}{2}}+\|E\langle x,x\rangle E^{\ast} \|^{\frac{1}{2}}\\
		&\leq \|B\|\|x\|+\|E\|\|x\|\\
		&\leq \big(\|B\|+\|E\|\big)\|x\|.
	\end{align*}
	Thus 
	\begin{equation}\label{eq5.2}
		\left\|\int_{\Omega}\langle x,(F_{w}+\Gamma_{w})x\rangle\langle (F_{w}+\Gamma_{w})x,x\rangle d\mu(w)\right\|\leq\big(\|B\|+\|E\|\big)^{2}\|x\|^{2}.
	\end{equation}
	On the other hand, 
	\begin{align*}
		\left\|\int_{\Omega}\langle x,(F_{w}+\Gamma_{w})x\rangle\langle (F_{w}+\Gamma_{w})x,x\rangle d\mu(w)\right\|^{\frac{1}{2}}&=\|\{(F_{w}+\Gamma_{w})x\}_{w\in\Omega}\|\\
		&\geq \|\{F_{w}x\}_{w\in\Omega}\|-\|\{\Gamma_{w}x\}_{w\in\Omega}\|\\
		&\geq \left\|\int_{\Omega}\langle x,F_{w}x\rangle\langle F_{w}x,x\rangle d\mu(w)\right\|^{\frac{1}{2}}-\\
		&\qquad\left\|\int_{\Omega}\langle x,\Gamma_{w}x\rangle\langle \Gamma_{w}x,x\rangle d\mu(w)\right\|^{\frac{1}{2}}\\
		&\geq \|A^{-1}\|^{-1}\|x\|-\|E\|\|x\|\\
		&\geq (\|A^{-1}\|^{-1}-\|E\|)\|x\|. 
	\end{align*}
	Hence 
	\begin{equation}\label{eq5.3}
		(\|A^{-1}\|^{-1}-\|E\|)\|x\|\leq \left\|\int_{\Omega}\langle x,(F_{w}+\Gamma_{w})x\rangle\langle (F_{w}+\Gamma_{w})x,x\rangle d\mu(w)\right\|^{\frac{1}{2}}. 
	\end{equation}
	Therefore, from \eqref{eq5.2} and 
	\eqref{eq5.3}, $\{(F_{w}+\Gamma_{w})\}_{w\in\Omega}$ is a $\ast$-continuous frame for $\mathcal{H}$.
\end{proof}
\begin{theorem}
	Let $\{T_{w}\}_{w \in \Omega}$ be a $\ast$-continuous frame  for $End_{\mathcal{A}}^{\ast}(\mathcal{H})$  with respect to $(\Omega, \mu)$ and with bounds $A$ and $B$, let $\{R_{w}\}_{w \in \Omega} \subset End_{\mathcal{A}}^{\ast}(\mathcal{H})$ and $\{\alpha_{w}\}_{w \in \Omega},\{\beta_{w}\}_{w \in \Omega} \in \mathbb{R}$ be two positively  family.
	If there exist two constants $0\leq \lambda, \mu<1$ such that for any $x \in\mathcal{H} $ we have 
	\begin{align*}
		\|\int_{\Omega}&\langle  x,(\alpha_{w}T_{w}-\beta_{w}R_{w}) x\rangle_{\mathcal{A}}\langle  (\alpha_{w}T_{w}-\beta_{w}R_{w})x, x\rangle_{\mathcal{A}} d\mu({\omega})\|^{\frac{1}{2}}\leq \\
		& \lambda \|\int_{\Omega}\langle  x,\alpha_{w}T_{w}x\rangle_{\mathcal{A}}\langle  \alpha_{w}T_{w}x,x\rangle_{\mathcal{A}} d\mu({\omega})\|^{\frac{1}{2}}
		+\mu\|\int_{\Omega}\langle  x,\beta_{w}R_{w} x\rangle_{\mathcal{A}}\langle \beta_{w}R_{w} x, x\rangle_{\mathcal{A}} d\mu({\omega})\|^{\frac{1}{2}}
	\end{align*}
	Then  $\{R_{w}\}_{w \in \Omega}$ is a $\ast$-continuous  frame for $End_{\mathcal{A}}^{\ast}(\mathcal{H})$ with respect to $(\Omega, \mu)$.
\end{theorem}
\begin{proof}
	For every $x \in \mathcal{H}$, we have
	\begin{align*}
		\|\{\beta_{w}R_{w}x\}_{w \in \Omega}\|&\leq \|\{(\alpha_{w}T_{w}-\beta_{w}R_{w})x\}_{w \in \Omega}\| +\|\{\alpha_{w}T_{w}x\}_{w\in \Omega}\|\\
		&\leq \mu \|\{\beta_{w}R_{w}x\}_{w\in \Omega}\| + \lambda \|\{\alpha_{w}T_{w}x\}_{w\in \Omega}\|+ \|\{\alpha_{w}T_{w}x\}_{w\in \Omega}\|\\
		&=(1+\lambda) \|\{\alpha_{w}T_{w}x\}_{w\in \Omega}\|+ \mu \|\{\beta_{w}R_{w}x\}_{w \in \Omega}\|.
	\end{align*}
	Then,
	$$(1-\mu)\|\{\beta_{w}R_{w}x\}_{w\in \Omega}\|\leq(1+\lambda) \|\alpha_{w}T_{w}x\|. $$
	Therefore
	$$(1-\mu)\inf_{\omega \in \Omega} (\beta_{w})\|\{R_{w}x\}_{w \in \Omega}\|\leq(1+\lambda) \sup_{\omega \in \Omega}(\alpha_{w})\|\{T_{w}x\}_{w\in \Omega}\|. $$
	Hence 
	$$\|\{R_{w}x\}_{w\in \Omega}\| \leq \frac{(1+\lambda) \sup_{\omega \in \Omega}(\alpha_{w})}{(1-\mu)\inf_{\omega \in \Omega} (\beta_{w})}\|\{T_{w}x\}_{w\in \Omega}\|.$$
	Also, for all $x \in \mathcal{H}$, we have 
	\begin{align*}
		\|\{(\alpha_{w}T_{w}x\}_{w\in \Omega}\| &\leq \|\{(\alpha_{w}T_{w}-\beta_{w}R_{w})x\}_{w \in \Omega}\| +\|\{\beta_{w}R_{w}x\}_{w \in \Omega}\|\\
		&\leq \mu \|\{\beta_{w}R_{w}x\}_{w \in \Omega}\| + \lambda \|\{\alpha_{w}T_{w}x\}_{w\in \Omega}\|+ \|\{\alpha_{w}T_{w}x\}_{w\in \Omega}\|\\
		&= \lambda \|\{\alpha_{w}T_{w}x\}_{w\in \Omega}\|+(1+\mu)\|\{\beta_{w}R_{w}x\}_{w\in \Omega}\|.
	\end{align*}
	then
	$$(1- \lambda)\|\{\alpha_{w}T_{w}x\}_{w\in \Omega}\| \leq (1+\mu)\|\{\beta_{w}R_{w}x\}_{w\in \Omega}\|.$$
	Hence 
	$$(1- \lambda)\inf_{\omega \in \Omega} (\alpha_{w})\|\{T_{w}x\}_{w\in \Omega}\| \leq (1+\mu)\sup_{\omega \in \Omega}(\beta_{w})\|\{R_{w}x\}_{w\in \Omega}\|.$$
	Thus
	$$\frac{(1- \lambda)\inf_{\omega \in \Omega} (\alpha_{w})}{(1+\mu)\sup_{\omega \in \Omega}(\beta_{w})}\|\{T_{w}x\}_{w\in \Omega}\| \leq \|\{R_{w}x\}_{w\in \Omega}\|.$$
	Therefore
	$$A(\frac{(1- \lambda)\inf_{\omega \in \Omega} (\alpha_{w})}{(1+\mu)\sup_{\omega \in \Omega}(\beta_{w})}) \|\langle x,x\rangle_{\mathcal{A}}\|(\frac{(1- \lambda)\inf_{\omega \in \Omega} (\alpha_{w})}{(1+\mu)\sup_{\omega \in \Omega}(\beta_{w})})A^{*}\leq \|\{R_{w}x\}_{w}\|^2.  $$
	So,  
	\begin{align*}
		\|\{R_{w}x\}_{w\in \Omega}\|^2&\leq (\frac{(1+ \lambda)\sup_{\omega \in \Omega}(\alpha_{w})}{(1-\mu)\inf_{\omega \in \Omega} (\beta_{w})})^2  \|\{T_{w}x\}_{w\in \Omega}\|^2 \\\leq
		& B (\frac{(1+ \lambda)sup(\alpha_{w})}{(1-\mu)\inf_{\omega \in \Omega} (\beta_{w})})  \|\langle x,x \rangle_{\mathcal{A}}\|(\frac{(1+ \lambda)sup(\alpha_{w})}{(1-\mu)\inf_{\omega \in \Omega} (\beta_{w})})B^{*} .  
	\end{align*}
	Hence 
	\begin{align*}
		&A(\frac{(1- \lambda) \inf_{\omega \in \Omega} (\alpha_{w})}{(1+\mu) \sup_{\omega \in \Omega}(\beta_{w})}) \|\langle x,x\rangle_{\mathcal{A}}\|(\frac{(1- \lambda) \inf_{\omega \in \Omega} (\alpha_{w})}{(1+\mu) \sup_{\omega \in \Omega}(\beta_{w})})A^{*}\\&\leq \|\int_{\Omega}\langle  x,R_{w}x \rangle_{\mathcal{A}}\langle R_{w} x,x \rangle_{\mathcal{A}} d\mu({\omega})\|\\
		&\leq  B (\frac{(1+ \lambda) \sup_{\omega \in \Omega}(\alpha_{w})}{(1-\mu) \inf_{\omega \in \Omega} (\beta_{w})})  \|\langle x,x \rangle_{\mathcal{A}}\|(\frac{(1+ \lambda) \sup_{\omega \in \Omega}(\alpha_{w})}{(1-\mu) \inf_{\omega \in \Omega} (\beta_{w})})B^{*} 
	\end{align*}
	This give that $\{R_{w}\}_{w \in \Omega}$ is a $\ast$-continuous frame for $End_{\mathcal{A}}^{\ast}(\mathcal{H})$ with respect to $(\Omega, \mu)$.
\end{proof}
\begin{theorem}
	Let $\{T_{w}\}_{w \in \Omega}$ be a $\ast$-continuous frame for $End_{\mathcal{A}}^{\ast}(\mathcal{H})$ with bounds $\nu$ and $\delta$. Let $\{R_{w}\}_{w \in \Omega} \in End_{\mathcal{A}}^{\ast}(\mathcal{H})$ and  $\alpha, \, \beta \geq 0$. If $0\leq \alpha + \frac{\beta}{\nu\nu^{\ast}}< 1$ such that for all $x\in \mathcal{H}$, we have  
	$$\|\int_{\Omega}\langle  x,(T_{w}-R_{w})x\rangle_{\mathcal{A}}\langle  (T_{w}-R_{w})x,x\rangle_{\mathcal{A}} d\mu({\omega})\|\leq \alpha \|\int_{\Omega}\langle  x,T_{w}  x\rangle_{\mathcal{A}}\langle T_{w} x,  x\rangle_{\mathcal{A}} d\mu({\omega})\| +\beta \|\langle  x, x\rangle_{\mathcal{A}}\|$$
	Then $\{R_{w}\}_{w \in \Omega}$ is a $\ast$-continuous frame  with bounds $\nu\left(1-\sqrt{\alpha +\frac{\beta}{\nu\nu^{\ast}}}\right)$ and $\delta\left(1+\sqrt{\alpha +\frac{\beta}{\nu\nu^{\ast}}}\right)$.
\end{theorem}
\begin{proof}
	Let $\{T_{w}\}_{w \in \Omega}$ be a $\ast$- continuous  fram with bounds $\nu$ and $\delta$. Then for any $x \in \mathcal{H}$, we have
	\begin{align*}
		\|\{T_{w}x\}_{w\in \Omega}\|&\leq \|\{(T_{w}-R_{w})x\}_{w\in \Omega}\| +\|\{R_{w}x\}_{w\in \Omega}\|\\
		&\leq (\alpha \|\int_{\Omega}\langle  x,T_{w}x\rangle_{\mathcal{A}}\langle  T_{w}x,x\rangle_{\mathcal{A}} d\mu({\omega})\| +\beta \|\langle  x, x\rangle_{\mathcal{A}}\|)^{\frac{1}{2}}\\
		& \quad +\|\int_{\Omega}\langle  x,R_{w} x\rangle_{\mathcal{A}}\langle  R_{w}x, x\rangle_{\mathcal{A}} d\mu({\omega})\|^{\frac{1}{2}}\\
		&\leq  (\alpha \|\int_{\Omega}\langle T_{w}x,x\rangle_{\mathcal{A}} d\mu({\omega})\| +\frac{\beta}{\nu\nu^{\ast}} \|\int_{\Omega}\langle x,T_{w}  x\rangle_{\mathcal{A}}\langle T_{w}x,  x\rangle_{\mathcal{A}} d\mu({\omega})\|)^{\frac{1}{2}}\\
		&\quad +\|\int_{\Omega}\langle  x,R_{w} x\rangle_{\mathcal{A}}\langle  R_{w}x, x\rangle_{\mathcal{A}} d\mu({\omega})\|^{\frac{1}{2}}\\
		&=\sqrt{\alpha +\frac{\beta}{\nu\nu^{\ast}}} \|\{T_{w}x\}_{w\in \Omega}\|+\|\int_{\Omega}\langle  x,R_{w} x\rangle_{\mathcal{A}}\langle R_{w}x, x\rangle_{\mathcal{A}} d\mu({\omega})\|^{\frac{1}{2}}.
	\end{align*}
	Therefore
	$$\left(1-\sqrt{\alpha +\frac{\beta}{\nu\nu^{\ast}}}\right)\|\{T_{w}x\}_{w\in \Omega}\| \leq \|\int_{\Omega}\langle  x,R_{w}x\rangle_{\mathcal{A}}\langle R_{w}x,x\rangle_{\mathcal{A}} d\mu({\omega})\|^{\frac{1}{2}}.  $$
	Thus 
	\begin{align*}
		\nu\left(1-\sqrt{\alpha +\frac{\beta}{\nu\nu^{\ast}}}\right)\|\langle  x, x\rangle_{\mathcal{A}}\|\left(1-\sqrt{\alpha +\frac{\beta}{\nu\nu^{\ast}}}\right)\nu^{\ast}
		&\leq  \|\int_{\Omega}\langle  x,R_{w} x\rangle_{\mathcal{A}}\langle R_{w} x, x\rangle_{\mathcal{A}} d\mu({\omega})\|.\\
	\end{align*}
	Also, we have 
	\begin{align*}
		\|\{R_{w}x\}_{w\in \Omega}\|&\leq \|\{(T_{w}-R_{w})x\}_{w\in \Omega}\| +\|\{T_{w}x\}_{w\in \Omega}\|\\
		&\leq \sqrt{\alpha +\frac{\beta}{\nu\nu^{\ast}}} \|\{T_{w}x\}_{w\in \Omega}\|+\|\{T_{w}x\}_{w\in \Omega}\|\\
		&= \left(1+\sqrt{\alpha +\frac{\beta}{\nu\nu^{\ast}}}\right) \|\{T_{w}x\}_{w\in \Omega}\|\\
		&\leq \sqrt{\delta}  \left(1+\sqrt{\alpha +\frac{\beta}{\nu\nu^{\ast}}}\right) \|\langle x,x\rangle_{\mathcal{A}}\|^{\frac{1}{2}}\sqrt{\delta^{\ast}}.
	\end{align*}
	Hence   
	$$\|\int_{\Omega}\langle  x,R_{w} x\rangle_{\mathcal{A}}\langle  R_{w}x, x\rangle_{\mathcal{A}} d\mu({\omega})\|\leq \delta \left(1+\sqrt{\alpha +\frac{\beta}{\nu\nu^{\ast}}}\right) \|\langle x,x\rangle_{\mathcal{A}}\|\left(1+\sqrt{\alpha +\frac{\beta}{\nu\nu^{\ast}}}\right)\delta^{\ast}.$$
	Therefore 
	\begin{align*}
		\small \nu\left(1-\sqrt{\alpha +\frac{\beta}{\nu\nu^{\ast}}}\right)\|\langle x, x\rangle_{\mathcal{A}}\| & \left(1-\sqrt{\alpha +\frac{\beta}{\nu\nu^{\ast}}}\right)\nu^{\ast} \leq \|\int_{\Omega}\langle  x,R_{w} x\rangle_{\mathcal{A}}\langle R_{w} x, x\rangle_{\mathcal{A}} d\mu({\omega})\|\\
		&\leq  \delta \left(1+\sqrt{\alpha +\frac{\beta}{\nu\nu^{\ast}}}\right) \|\langle x,x\rangle_{\mathcal{A}}\|\delta \left(1+\sqrt{\alpha +\frac{\beta}{\nu\nu^{\ast}}}\right) \delta^{\ast}.
	\end{align*}
	Hence $\{R_{w}\}_{w \in \Omega}$ is a $\ast$-continuous frame  with bounds $\nu\left(1-\sqrt{\alpha +\frac{\beta}{\nu\nu^{\ast}}}\right)$ and $\delta\left(1+\sqrt{\alpha +\frac{\beta}{\nu\nu^{\ast}}}\right)$.	
\end{proof}
\begin{corollary}
	Let $\{T_{w}\}_{w \in \Omega}$  is a $\ast$-continuous frame for $End_{\mathcal{A}}^{\ast}(\mathcal{H})$  with bounds $\nu$ and $\delta$. Let $\{R_{w}\}_{w \in \Omega} \subset End_{\mathcal{A}}^{\ast}(\mathcal{H}) $ and $0\leq \alpha $.
	If $0\leq \alpha <\nu$ such that 
	$$\|\int_{\Omega}\langle  x,(T_{w}-R_{w}) x \rangle_{\mathcal{A}}\langle  (T_{w}-R_{w})x, x \rangle_{\mathcal{A}} d\mu({\omega})\|\leq \alpha \|\langle x, x\rangle_{\mathcal{A}}\|,\,\ x \in \mathcal{H},$$
	then $\{R_{w}\}_{w \in \Omega}$ is a $\ast$- continuous  frame with bounds $\nu(1-\sqrt{{\frac{\alpha}{\nu\nu^{\ast}}}})^2$ and $\delta(1+\sqrt{{\frac{\alpha}{\nu\nu^{\ast}}}})^2$.
	
\end{corollary}
\begin{proof}
	The proof comes from the previous theorem.
\end{proof}
\begin{theorem}
	For $k=1,2,...,n$, let $\{T_{k,w}\} _{w \in \Omega} \subset End_{\mathcal{A}}^{\ast}(\mathcal{H})$ be a $\ast$-continuous operator frames with bounds  $\nu_k$ and  $\delta_k$ and let $\{R_{k,w}\} _{w \in \Omega}\subset End_{\mathcal{A}}^{\ast}(\mathcal{H})$.\\ Let $L: L^{2}(\Omega, \mathcal{A})\longrightarrow L^{2}(\Omega, \mathcal{A}) $ be a bounded linear operator such that
	\begin{equation*}
		L(\{\sum_{k=1}^{n}R_{k,w}\} _{w \in \Omega})=\{T_{p,w}\} _{w \in \Omega} \quad for some \quad p \in \{1,2,\dots,n\}.
	\end{equation*}
	If there exists a constant $\lambda >0$ such that for each $x\in \mathcal{H}$ and $ k=1,\dots,n$, we have
	\begin{equation*}
		\| \int_{\Omega}\langle  x,(T_{k,w}-R_{k,w}) x\rangle_{\mathcal{A}}\langle (T_{k,w}-R_{k,w}) x, x\rangle_{\mathcal{A}} d\mu(w)\| \leq \lambda \|\int_{\Omega}\langle  x,T_{k,w} x\rangle_{\mathcal{A}}\langle T_{k,w} x, x\rangle_{\mathcal{A}} d\mu(w)\|.
	\end{equation*}
	Then $\{\sum_{k=1}^{n}R_{k,w}\} _{w \in \Omega}$ is a $\ast$-continuous operator frame for $End_{\mathcal{A}}^{\ast}(\mathcal{H})$.
\end{theorem}
\begin{proof}
	For all $x \in \mathcal{H}$, we have
	\begin{align*}
		\|\{\sum_{k=1}^{n}R_{k,w}x\} _{w \in \Omega}\|&\leq \sum_{k=1}^{n}\|\{R_{k,w}x\} _{w \in \Omega}\|\\
		&\leq \sum_{k=1}^{n}(\|\{T_{k,w}-R_{k,w}x\}_{w \in \Omega}\|+\|\{T_{k,w}x\}_{w \in \Omega}\|)\\
		&\leq (1+\sqrt{\lambda}) \|\sum_{k=1}^{n} \|\{T_{k,w}x\} _{w \in \Omega}\|\\
		&\leq  (1+\sqrt{\lambda})\sum_{k=1}^{n}\sqrt{\delta_k} \|\langle  x, x\rangle_{\mathcal{A}}\|^{\frac{1}{2}}\sqrt{\delta^{*}_k}.
	\end{align*}
	Since, for any $x\in \mathcal{H}$, we have 
	$$\|L(\{\sum_{k=1}^{n}R_{k,w}\} _{w \in \Omega})\|=\|\{T_{p,w}\} _{w \in \Omega}\|$$
	Then
	\begin{align*}
		\sqrt{\nu_p}\|\langle  x, x\rangle_{\mathcal{A}}\|^{\frac{1}{2}}\sqrt{\nu^{*}_p}&\leq \|\{T_{p,w}\} _{w \in \Omega}\|\\
		&= \|L(\{\sum_{k=1}^{n}R_{k,w}\} _{w \in \Omega})\|\\
		&\leq \|L\| \|\{\sum_{k=1}^{n}R_{k,w}\} _{w \in \Omega}\|.
	\end{align*}
	Hence
	$$\frac{\sqrt{\nu_p}}{\|L\|}\|\langle  x, x\rangle_{\mathcal{A}}\|^{\frac{1}{2}}{\nu^{*}_p}\leq \|\{\sum_{k=1}^{n}R_{k,w}\} _{w \in \Omega}\|,\qquad  x\in \mathcal{H}.  $$
	Therefore 
	$$\frac{\sqrt{\nu_p}}{\|L\|}\|\langle  x, x\rangle_{\mathcal{A}}\|^{\frac{1}{2}}{\nu^{*}_p}\leq \|\{\sum_{k=1}^{n}R_{k,w}\} _{w \in \Omega}\|\leq (1+\sqrt{\lambda})\sum_{k=1}^{n}\sqrt{\delta_k} \|\langle  x, x\rangle_{\mathcal{A}}\|^{\frac{1}{2}}\sqrt{\nu^{*}_p}.$$
	This give that $\{\sum_{k=1}^{n} R_{k,w}\} _{w \in \Omega}$ is a $\ast$-continuous operator frame for $End_{\mathcal{A}}^{\ast}(\mathcal{H})$.	
\end{proof}
\begin{theorem}
	Let  $\{\Lambda_{w }: w\in\Omega\}$ is a continuous frame for Hilbert $C^{\ast}$-module $\mathcal{H}$ with bound $A$, $B$ and $\{\Gamma_{w}: w\in\Omega\}$  be a family of operators, such that  $\{\Gamma_{w}: w\in\Omega\}$ is strongly measurable for each $x \in \mathcal{H}$. 
	If there exist constants $\lambda_{1},\lambda_{2},\gamma \geq 0,$ such that $max\{\lambda_{2},\frac{\gamma}{A}+\lambda_{1} \}<1$ and for each $x,y \in \mathcal{H}$:	
	\begin{multline*}
		\lVert\int_{\Omega}\langle(\Lambda^{\ast}_{w}\Lambda_{w}-\Gamma^{\ast}_{w}\Gamma_{w})x,y\rangle d\mu(w)\lVert \leq \\
		\lambda_{1}\lVert\int_{\Omega}\langle\Lambda^{\ast}_{w}\Lambda_{w}x,y\rangle d\mu(w)\lVert+\lambda_{2}\lVert\int_{\Omega}\langle\Gamma^{\ast}_{w}\Gamma_{w}x,y\rangle d\mu(w)\lVert+\gamma\lVert x\lVert^{2}.
	\end{multline*}
	Then
	$\{ \Gamma_{w }\in End_{\mathcal{A}}^{\ast}(U,V_{w}): w\in\Omega\}$ is a continuous frame for Hilbert $C^{\ast}$-module $U$ with respect to $\{V_{w}: w\in \Omega\} $ with bounds:
	$\frac{(1-\lambda_{1})A-\gamma}{1+\lambda_{2}}$
	and $\frac{(1+\lambda_{1})B+\gamma }{1-\lambda_{2}}$
\end{theorem}
\begin{proof}
	For each $x,y \in \mathcal{U}$
	\begin{align*}
		\lVert\int_{\Omega}\langle\Gamma^{\ast}_{w}\Gamma_{w}x,y\rangle d\mu(w)\lVert &\leq\lVert\int_{\Omega}\langle(\Lambda^{\ast}_{w}\Lambda_{w}-\Gamma^{\ast}_{w}\Gamma_{w})x,y\rangle d\mu(w)\lVert+
		\lVert\int_{\Omega}\langle\Lambda^{\ast}_{w}\Lambda_{w}x,y\rangle d\mu(w)\lVert\\
		&\leq 	(1+\lambda_{1})\lVert\int_{\Omega}\langle\Lambda^{\ast}_{w}\Lambda_{w}x,y\rangle d\mu(w)\lVert+\lambda_{2}\lVert\int_{\Omega}\langle\Gamma^{\ast}_{w}\Gamma_{w}x,y\rangle d\mu(w)\lVert+\gamma\lVert x\lVert^{2}\\	
	\end{align*}
	For each $x,y \in \mathcal{U}$,
	$$\lVert\int_{\Omega}\langle\Gamma^{\ast}_{w}\Gamma_{w}x,y\rangle d\mu(w)\lVert\leq 	\frac{(1+\lambda_{1})}{1-\lambda_{2}}\lVert\int_{\Omega}\langle\Lambda^{\ast}_{w}\Lambda_{w}x,y\rangle+\frac{\gamma}{1-\lambda_{2}}\lVert x\lVert^{2}$$
	Therefore, for each $x,y \in \mathcal{U}$,
	$$\lVert\int_{\Omega}\langle\Gamma^{\ast}_{w}\Gamma_{w}x,y\rangle d\mu(w)\lVert\leq 	\frac{(1+\lambda_{1})}{1-\lambda_{2}}B\lVert\langle x,x\rangle \lVert +\frac{\gamma}{1-\lambda_{2}}\langle x,x\rangle \lVert$$
	Hence,for each $x \in \mathcal{U}$
	\begin{align*}
		\lVert\int_{\Omega}\langle\Gamma_{w}x,\Gamma_{w}x\rangle d\mu(w)\lVert&\leq 	(\frac{(1+\lambda_{1})}{1-\lambda_{2}}B +\frac{\gamma}{1-\lambda_{2}})\lVert\langle x,x\rangle \lVert\\
		&\leq 	(\frac{(1+\lambda_{1})B+\gamma}{1-\lambda_{2}} )\lVert\langle x,x\rangle \lVert
	\end{align*}
	Therefore, $\{\Gamma_{w}: w\in\Omega\}$ is a Continuous Bessel family for Hilbert $C^{\ast}$-module $U$.\\
	Now, we show that $\{\Gamma_{w}: w\in\Omega\}$ has the lower continuous frame condition.\\
	\begin{align*}
		\lVert\int_{\Omega}\langle\Gamma^{\ast}_{w}\Gamma_{w}x,x\rangle d\mu(w)\lVert &=\lVert\int_{\Omega}\langle(\Lambda^{\ast}_{w}\Lambda_{w}-\Gamma^{\ast}_{w}\Gamma_{w})x,x\rangle d\mu(w)+
		\int_{\Omega}\langle\Lambda^{\ast}_{w}\Lambda_{w}x,x\rangle d\mu(w)\lVert\\	
		&\geq 
		\int_{\Omega}\langle\Lambda^{\ast}_{w}\Lambda_{w}x,x\rangle d\mu(w)\lVert-\lVert\int_{\Omega}\langle(\Lambda^{\ast}_{w}\Lambda_{w}-\Gamma^{\ast}_{w}\Gamma_{w})x,x\rangle d\mu(w)\\
		&\geq (1- \lambda_{1})\lVert\int_{\Omega}\langle\Lambda^{\ast}_{w}\Lambda_{w}x,x\rangle d\mu(w)\lVert-\lambda_{2}\lVert\int_{\Omega}\langle\Gamma^{\ast}_{w}\Gamma_{w}x,x\rangle d\mu(w)\lVert-\gamma\lVert x\lVert^{2}	
	\end{align*}
	Hence, for each $x \in \mathcal{U}$
	\begin{align*}
		\lVert\int_{\Omega}\langle\Gamma^{\ast}_{w}\Gamma_{w}x,x\rangle d\mu(w)\lVert&\geq 	(\frac{(1-\lambda_{1})}{1+\lambda_{2}}A -\frac{\gamma}{1+\lambda_{2}})\lVert\langle x,x\rangle \lVert\\
		&\geq 	(\frac{(1-\lambda_{1})A-\gamma}{1+\lambda_{2}} )\lVert\langle x,x\rangle \lVert
	\end{align*}
	Therefore, 	$\{ \Gamma_{w }\in End_{\mathcal{A}}^{\ast}(U,V_{w}): w\in\Omega\}$ is a continuous frame for Hilbert $C^{\ast}$-module $U$ with respect to $\{V_{w}: w\in \Omega\} $ with bounds:
	$\frac{(1-\lambda_{1})A-\gamma}{1+\lambda_{2}}$
	and $\frac{(1+\lambda_{1})B+\gamma}{1-\lambda_{2}}$.
\end{proof}
\section{The dual ${\ast}$-continuous frames}
\begin{definition}
	Let $\left\{F_{w} : w \in \Omega \right\}$ be a ${\ast}$-continuous frame for $\mathcal{H}$ with ${\ast}$-continuous frame operator $S$. If there exists a ${\ast}$-continuous frame $\left\{G_{w} \in \mathcal{H}: w \in \Omega\right\}$ for $\mathcal{H}$ such that $F=\int_{\Omega}\left\langle F, G_{w}\right\rangle F_{w}$ for $f \in \mathcal{H}$, then the $*$continuous frame $\left\{G_{w}\right\}_{w \in \Omega}$ is called the dual ${\ast}$-continuous frame of $\left\{F_{w}\right\}_{w \in \Omega}$. The spacial dual ${\ast}$-continuous frame $\left\{S^{-1} F_{w}\right\}_{w \in \Omega}$ is said to be the canonical dual ${\ast}$-continuous frame of $\left\{F_{w}\right\}_{w \in \Omega}$.
\end{definition}	
It is well known, that if $T$ and $V$ are pre-* continuous frame operators of two $*$-Bessel sequences $\left\{F_{w}\right\}_{w \in \Omega}$ and $\left\{G_{w}\right\}_{w \in \Omega}$, respectively, then $F=\int_{\Omega}\left\langle F, G_{w}\right\rangle F_{w}$ for $f \in \mathcal{H}$ if and only if $T^{*} V=$ $Id_{\mathcal{H}}$. The following lemma shows that the roles of two $*$- continuous Bessel sequences can be changed and obtains a relation between bounds of $\left\{F_{w}\right\}_{w \in \Omega}$ and $\left\{G_{w}\right\}_{w \in \Omega}$.

\begin{lemma}\label{Lemma7}
	Let $\left\{F_{w}\right\}_{w \in \Omega}$ and $\left\{G_{w}\right\}_{w \in \Omega}$ be ${\ast}$-continuous Bessel sequences for $\mathcal{H}$ with pre-${\ast}$-continuous frame operators $T$ and $V$, respectively. Then for $x \in \mathcal{H}$ the following statements are equivalent.
	\begin{enumerate}
		\item [i.] $x=\int_{\Omega}\left\langle x, g_{w}\right\rangle F_{w}$.
		\item [ii.] $x=\int_{\Omega}\left\langle x, F_{w}\right\rangle g_{w}$.
	\end{enumerate}
	In the case that one of the above equalities is satisfied, $\left\{F_{w}\right\}_{w \in \Omega}$ and $\left\{G_{w}\right\}_{w \in \Omega}$ are dual ${\ast}$-continuous frames.
	Moreover, if $B$ is an upper ${\ast}$-continuous frame bound for $\left\{F_{w}\right\}_{w \in \Omega}$ and $S$ is its ${\ast}$-continuous frame operator, then $B\left\|S^{-1}\right\|^{-\frac{1}{2}}\|T\|^{-1}$ is a lower ${\ast}$-continuous frame bound for $\left\{G_{w}\right\}_{w \in \Omega}$.
\end{lemma}
\begin{proof}
	Suppose that the conditions $i$ and $i i$ are valid. By $i$, we have $T^{*} V=i d_{\mathcal{H}}$ and $T^{*}$ is surjective. 
	Then it follows that the sequence $\left\{F_{w}\right\}_{w \in \Omega}$ is a ${\ast}$-continuous frame. Similarly, the ${\ast}$-continuous Bessel sequence $\left\{G_{w}\right\}_{w \in \Omega}$ is a ${\ast}$-continuous frame.
\end{proof}	
Finally, let $B$ be an upper ${\ast}$-continuous frame bound for $\left\{F_{w}\right\}_{w \in \Omega}$. By the definition of ${\ast}$-continuous frames $\left\{F_{w}\right\}_{w \in \Omega}$ and $T^{*} V=i d_{\mathcal{H}}$, we can write $$\langle T x, T x\rangle \leq B\left\langle T^{*} V x, T^{*} V x\right\rangle B^{*} for x \in \mathcal{H}.$$  
Using Lemma \ref{Lemma0-1}, we have $$\left\|\left(T^{*} T\right)^{-1}\right\|^{-1}\langle x, x\rangle \leq\langle T x, T x\rangle.$$ For $x \in \mathcal{H}$. It follows that $$B^{-1}\left\|S^{-1}\right\|^{-\frac{1}{2}}\|T\|^{-1}\langle x, x\rangle\left(B^{-1}\left\|S^{-1}\right\|^{-\frac{1}{2}}\|T\|^{-1}\right)^{*} \leq\langle V x, V x\rangle, \qquad x \in \mathcal{H}. $$ 
Therefore, $B\left\|S^{-1}\right\|^{-\frac{1}{2}}\|T\|^{-1}$ is a lower ${\ast}$-continuous frame bound for $\left\{G_{w}\right\}_{w \in J\Omega}$ and the proposition follows.	

\begin{proposition}
	Let $\left\{x_{w}\right\}_{w \in \Omega}$ be a sequence of a finitely or countably generated Hilbert $\mathcal{A}$-module $\mathcal{H}$ over a unital $C^{*}$-algebra $\mathcal{A}$. Then $\left\{x_{w}\right\}_{w \in \Omega}$ is a Bessel sequence with Bessel bound $D$ if and only if the operator $T: L^{2}(\Omega, \mathcal{A}) \rightarrow \mathcal{H}$ defined $b y$
	$
	T\left\{c_{w}\right\}_{w \in \Omega}=\int_{\Omega} c_{w} x_{w}d\mu(w)
	$
	is a well-defined bounded operator from $L^{2}(\Omega, \mathcal{A})$ into $\mathcal{H}$ with $\|T\| \leq \sqrt{D}$.
\end{proposition}
\begin{proof}
	" $\Rightarrow "$. Suppose that $\left\{x_{w}\right\}_{w \in \Omega}$ is a Bessel sequence with bound $D$.\\
	To see the boundedness of $T$, we consider
	\begin{align*}
		\left\|T\left\{c_{w}\right\}\right\|^{2} &=\sup _{\|x\|=1}\left\|\left\langle T\left\{c_{w}\right\}, x\right\rangle\right\|^{2} \\
		&=\sup _{\|x\|=1}\left\|\int_{\Omega} c_{w}\left\langle x_{w}, x\right\rangle d\mu(w)\right\|^{2}\\
		&\leq \sup _{\|x\|=1}\left\|\int_{\Omega}\left\langle x, x_{w}\right\rangle\left\langle x_{w}, x\right\rangle d\mu(w)\right\| \cdot\left\|\int_{\Omega} c_{w} c_{w}^{*}d\mu(w)\right\|\\ 
		&\leq D\left\|\int_{\Omega} c_{w} c_{w}^{*}d\mu(w)\right\|=D\left\|\left\{c_{w}\right\}\right\|^{2}
	\end{align*}
	This yields that $\|T\| \leq \sqrt{D}$.\\
	$" \Leftarrow "$. For arbitrary $x \in \mathcal{H}$ and $\left\{c_{w}\right\}_{w \in \Omega} \in L^{2}(\Omega, \mathcal{A})$, we have
	\begin{equation}\label{equation}
		\left\langle x, T\left\{c_{w}\right\}\right\rangle=\left\langle x, \int_{\Omega} c_{w} x_{w}d\mu(w)\right\rangle =\int_{\Omega}\left\langle x, x_{w}\right\rangle c_{w}^{*}d\mu(w)
	\end{equation}
	we see that $\left\{\left\langle x, x_{w}\right\rangle\right\}_{w \in \Omega} \in L^{2}(\Omega, \mathcal{A})$.
	From \eqref{equation}, we get
	$$
	\left\langle x, T\left\{c_{w}\right\}\right\rangle=\left\langle\left\{\left\langle x, x_{w}\right\rangle\right\},\left\{c_{w}\right\}\right\rangle
	$$
	which implies that $T$ is adjointable, with $T^{*} x=\left\{\left\langle x, x_{w}\right\rangle\right\}_{w \in \Omega}$, and hence $T$ is bounded.
	Note that
	$
	\left\|T^{*} x\right\|^{2}=\left\|\int_{\Omega}\left\langle x, x_{w}\right\rangle\left\langle x_{w}, x\right\rangle d\mu(w)\right\| \leq D\|\langle x, x\rangle\|=D\|x\|^{2}
	$
	Consequently, $\|T\|=\left\|T^{*}\right\| \leq \sqrt{D}$, as desired.
\end{proof}

\begin{proposition}
	Let $\left\{F_{w} : w \in \Omega\right\}$ be a ${\ast}$-continuous frame for $\mathcal{H}$ with pre-${\ast}$-continuous frame operator $T$. The set of all of ${\ast}$-continuous dual frames for $\left\{F_{w}\right\}_{w \in \Omega}$ is precisely the set of the families $\left\{g_{w}\right\}_{w \in \Omega}=$ $\left\{V^{*}\left(e_{w}\right)\right\}_{w \in \Omega}$, where $V: \mathcal{H} \longrightarrow L^{2}(\Omega, \mathcal{A})$ is an adjointable right-inverse of $T^{*}$ and the sequence $\left\{e_{w}\right\}_{w \in \Omega}$ is the standard basis for $L^{2}(\Omega, \mathcal{A})$.
\end{proposition}
\begin{proof}
	By Lemma \ref{Lemma7} and \cite[Proposition 3.11]{Thesis-Jing-2006}, the proof is clear.
\end{proof}	

\section*{Acknowledgments}
It is our great pleasure to thank the referee for his careful reading of the paper and for several helpful suggestions.

\end{document}